\documentclass[11pt]{article}
\usepackage {amssymb,latexsym}
\usepackage {amsmath}
\usepackage[T1]{fontenc}

\usepackage[english]{babel}
\usepackage[utf8]{inputenc}

\usepackage{verbatim}
\usepackage{amsthm}
\usepackage{bm}

\def\N{\mathbb{N}}

     \def\Hom{{\rm Hom}}
\def\ad{{\rm ad}}

\newcommand{\CCinf}{\mathcal{C}^\infty}
\newcommand{\ph}{[[\lambda]]}

\usepackage{xypic}
\xyoption{all}
\input{xypic}
\newdir{ >}{{}*!/-8pt/\dir{>}}

\newtheorem{theorem}{Theorem}[section]
\newtheorem{prop}{Proposition}[section]

\newtheorem{defi}{Definition}[section]

\newtheorem{rem}{Remark}[section]
\medskip

\makeatletter
\@addtoreset{equation}{section}

\makeatother

\usepackage{graphicx}
\newcommand{\naturalto}{%
	\mathrel{\vbox{\offinterlineskip
			\mathsurround=0pt
			\ialign{\hfil##\hfil\cr
				\normalfont\scalebox{1.2}{.}\cr
				$\longrightarrow$\cr}
	}}%
}

\begin{document}
	
	\thispagestyle{empty}
	\pagestyle{empty}
	\title
	{\LARGE \bf Noncommutative localization in 
		smooth deformation quantization}

	
	\author{
		Martin Bordemann, \\
		\texttt{Martin.Bordemann@uha.fr},\\[2mm]
		Universit\'{e} de Haute Alsace, Mulhouse, France\\[2mm]
		Benedikt Hurle, \\
		\texttt{Benedikt.Hurle@uha.fr}, \\[2mm]
		Chern Institute of Mathematics, Nankai University, Tianjin, P.R China\\[2mm]
		Hamilton Menezes de Araujo, \\
		\texttt{Hamilton.Araujo@uha.fr}, 	\texttt{Hamilton.Araujo@unilasalle.fr}\\
	[2mm]
		Universit\'{e} de Haute Alsace, Mulhouse, France\\
		Institut Polytechnique UniLaSalle, Beauvais, France\\[2mm]
		    Second Version
	}
	

	\maketitle
	\thispagestyle{empty}
	
	\vspace{1cm}
	\small
	\begin{center}
		\noindent {\bf Abstract}\\[5mm]
		\thispagestyle{empty}
		\begin{minipage}{13cm}
			In this paper we shall show the equivalence 
			of analytic localization and noncommutative algebraic localization (in the general and in
			Ore's sense) for algebras of smooth deformation
			quantization for several situations. The proofs are based on 
			old work by Whitney, Malgrange and
			Tougeron on the commutative algebra of smooth function 
			rings from the 60's and 70's. We comment
			on the more general situation of localizations of differential star products of commutative algebras equipped with a multiplicative subset.
		\end{minipage}
	\end{center}
	\normalsize
	
	\newpage
	
	\tableofcontents
	\thispagestyle{empty}
	\newpage
	
	\pagestyle{plain}
	
	\section*{Introduction}
	\addcontentsline{toc}{section}{Introduction}
	
Localization in  commutative algebra means a universal construction 
where a set of chosen elements in a given commutative
ring is made invertible (they will become denominators): 
the outcome is called a ring of fractions.
The classical example is the 
well-known passage from the integers to the field of rational
numbers. It is a very important tool in algebraic 
and analytical geometry. In differential geometry, however,
localization is rather used in the analytic sense, 
i.e.~the passage from globally defined smooth functions to those
which are only defined on an open subset. It follows from
the classical works by Whitney, Malgrange \cite{Mal67} and Tougeron \cite{Tou72}
that these analytical localizations are often isomorphic to
certain algebraic localizations in the smooth (or 
even $\mathcal{C}^k$, $k\in\mathbb{N}$) case, see also
\cite{FS98}, \cite{RS94} and the book \cite{NS03}.

For noncommutative algebras the problem of algebraic
localization has two solutions: a general construction (see e.g.~\cite[p.289]{Lam99}), and the better known solution
	initiated by \O{}.Ore \cite{Ore31} in the 1930's
	requiring additional conditions on the multiplicative
	subset, the famous Ore conditions. It turns out
	that the general construction is rather inexplicit
	and in some situations not very practical. On the other hand, the more particular Ore localization shares
	almost all properties of the commutative localization.
	
	In this work we would like to study noncommutative localization of algebras arising in deformation quantization:
	in this theory --invented by \cite{BFFLS78} in 1978--
	formal associative deformations of the algebra of all
	smooth complex-valued functions on a Poisson manifold,
	so-called star products, are studied aiming at an interpretation of the noncommutative
	multiplication of operators used in quantum mechanics.
	It is well-known that the first order commutator of such
	a deformation always gives rise to a Poisson bracket, but it
	is highly non-trivial to show that every Poisson bracket
	arises as a first order commutator of a deformation: this
	latter result is the famous Kontsevich formality Theorem,
	\cite{Kon03}.
	We consider star products given by formal power series of
	bidifferential operators (as almost every-one): these
	multiplications immediately define star products of
	locally defined functions by suitable `restrictions'.
	We mention that noncommutative localization has been
	used for the algebras in noncommutative geometry to
	describe inverses of functions appearing in passing to 
	coordiates, see e.g.~the recent work \cite{ACH16},
	\cite{AW17a}, and \cite{AW17b}.\\
	In this article we have chosen the algebra of smooth functions on a smooth manifold and not a framework
	of algebraic or analytical geometry which would have required a sheaf theoretic approach: firstly, historically deformation quantization has been formulated in a	differential geometry framework
	and has farther reaching existence and uniqueness results for smooth functions, secondly, in the smooth world there is no urgent need to pass to sheaves which simplifies the
	exposition, and thirdly, the commutative
	algebra of smooth function algebras seems to be 
	of a different, `funnier' kind (not Noetherian) which
	we liked to rediscover from the classical literature,
	in particular from the books \cite{Tou72} and \cite{NS03}.\\
	We first show that this analytical localization 
	of star product algebras is 
	isomorphic to the algebraic localization
	with respect to the set of all those formal power series
	of smooth functions whose zeroth order term does nowhere
	vanish on the given open set. As a by-product we have the
	result that this multiplicative set satisfies the right
	(and also left) Ore condition.\\
	In a similar way we can show that the set of all 
	germs of a star product algebra at a given point of the
	manifold --defined in analytic terms-- is isomorphic to the noncommutative localization
	of the complement of the maximal ideal of all those
	formal power series of functions whose term of order
	zero vanishes at the point.\\
	We also sketch a more general algebraic framework inspired by the two preceding results: given a commutative associative unital algebra $A$ over a commutative ring $K$, a multiplicative subset $S_0$ of $A$, and a bidifferential star product $\star$ (the bidifferential operators are defined in the well-known algebraic sense) then the following two constructions can be compared: first, localizing first the bidifferential operators \`{a} la G.Vezzosi \cite{Vez97} there is
	a star product $\star_{S_0}$ on $A_{S_0}[[\lambda]]$, the formal power series whose terms are in the localized algebra $A_{S_0}$, and secondly,
	considering first the natural `deformation of $S_0$, 
	$S=S_0+\lambda A[[\lambda]]$, which is a noncommutative multiplicative subset of the star product algebra
	$\big(A[[\lambda]],\star\big)$ there is the noncommutative
	localization (a priori in the general sense) of
	$A[[\lambda]]$ with respect to the deformation $S$ of $S_0$. Two quite natural questions arise: `\emph{Does localization commute with deformation ?}' --on which
	we give a positive answer in case $S_0$ has a sort of
	`common multiple property for sequences'--and
	`\emph{Is $S$ right or left Ore ?}' for which we give 
	an elementary
	counterexample in Section \ref{SecNonOreExample}.

	The paper is organised as follows: in the first section we
	recall some basic concepts of the commutative algebra of smooth function algebras (where we could not resist 
	the pleasure of the completely unnecessary study of its prime ideals) following Tougeron's book \cite{Tou72}, 
	and of (non)commutative localization
	following Lam's very nice text-book \cite{Lam99}. 
	\\
	In Section 2 we show the first localization result concerning
	open sets: an important tool is Tougeron's \emph{fonction aplatisseur} \cite{Tou65} which makes a given sequence of locally defined smooth
	functions globally defined by multiplication with a single suitable function being
	nowhere zero on the given open set. In this proof, we had to
	make explicit use of the seminorms defining the
	Fr\'{e}chet topology of the smooth
	function space.\\
	In Section 3 we prove a similar result for germs, heavily relying on the first theorem. Note that
	there is slight, but important difference between `germs over the formal
	power series ring $\mathbb{K}[[\lambda]]$' which we describe and the `formal power series of germs'
	on which we comment in Section
	\ref{SecCommutativelyLocalizedStarProducts}.\\
	Section 4 is devoted to the above-mentioned discussion of commutative localization
	of multidifferential operators (defined in the algebraical way) for any commutative
	algebra $A$ (over some unital commutative ring $K$)
	by general commutative 
	multiplicative subsets $S_0$ and its comparison with a `natural' noncommutative localization with respect to
	$S=S_0+\lambda A[[\lambda]]$. The technical tool
	will be a localization theorem of algebraic differential operators due to G.Vezzosi, 1997, \cite{Vez97}.\\
	In Section 5 we describe a simple example of a multiplicative set of the 
	type $S=S_0+\lambda A[[\lambda]]$ in smooth deformation quantization of the plane which is not Ore.\\

	\subsubsection*{Acknowledgements}
	
	The authors would like to thank Joakim Arnlind, Alberto Elduque, Jens Hoppe, Camille Laurent-Gengoux, Daniel Panazzolo, Leonid Ryvkin, 
	Zoran \v{S}koda, and Friedrich Wagemann for many valuable discussions and hints.

	
	\section{Review of basic concepts}
	
	Let $K$ be a fixed commutative associative unital ring such that
	$1=1_K\neq 0=0_K$.
	All $K$-algebras are supposed to be associative and unital. 
	We shall include unital  $K$-algebras $R$ isomorphic to $\{0\}$ for which $1_R=0_R$. Note that associative
	unital rings are always $\mathbb{Z}$-algebras in a natural
	way. In order to avoid clumsy notation we shall not
	write $1_R$, $1_K$, $0_R$ or $0_K$, but simply $1$ and $0$
	where the precise interpretation should be clear from the
	context.

	\subsection{Review of Commutative Algebra for smooth function algebras}
	
	\subsubsection{Elementary features of function algebras}
	\label{SubSubSecElementary features of function algebras}
	
	For the convenience of the reader (working in differential geometry) we shall give the following
	elementary survey which can be ignored on first reading. For more information see e.g.~\cite{BouCommAlg}.\\
	Recall some elementary commutative algebra of unital $K$-algebras of functions on a set: let
	$X$ be a set, $K$ be a commutative domain, and $R$ be a given unital subalgebra
	of the $K$-algebra of all the functions $X\to K$. As usual, for any subset
	$Y\subset X$, let $I(Y)\subset R$ be the set of functions in $R$ vanishing on
	$Y$ which is always an ideal of $R$ (the \emph{vanishing ideal of $Y$}), and for any subset $J\subset R$ let
	$Z(J)\subset X$ be the subset of those points of $X$ on which all functions in $J$ vanish. Clearly $Y\subset Z(I(Y))$ and $J\subset I(Z(J))$. Moreover for any two ideals $I_1$ and $I_2$ of $R$ it follows
	that $Z(I_1)\cup Z(I_2)= Z(I_1I_2)$ (here the fact that $K$ is a domain is used), hence the set of all subsets
	$Z(I)$, $I$ ideal of $R$, satisfies the axioms of the closed sets of
	a topology on $X$ called the \emph{Zariski topology on $X$ w.r.t.~$R$}.
	These closed subsets could be called \emph{$R$-algebraic sets}: in the particular
	case where $K=\mathbb{K}$ is a field, $X=\mathbb{K}^n$, and $R$ the polynomial ring $\mathbb{K}[x_1,\ldots,x_n]$
	these are the \emph{algebraic subsets}, whereas for $R$ being the algebra
	of analytic functions (for $\mathbb{K}=\mathbb{R}$
	or $\mathbb{K}=\mathbb{C}$) these sets are called 
	\emph{analytic subsets}. It is well-known that for
	$X=\mathbb{R}^n$ and $R$ the ring of smooth real-valued functions the Zariski topology of $X=\mathbb{R}^n$
	coincides with its usual topology. Returning to the general situation, note that the Zariski closure of 
	a set $Y\subset X$ is equal to $Z(I(Y))$. On the other hand
	the inclusion $J\subset I(Z(J))$ can sometimes be made more precise
	by what is called a
	\emph{Nullstellensatz}: in algebraic geometry over algebraically closed fields $I(Z(J))$ is equal to the
	set of all polynomials in $R$ such that a certain power
	is in $J$.\\
	Recall that an ideal $I\subset R$ is called \emph{proper} iff
	$I\neq R$ iff $1\not\in I$. Moreover, a proper ideal $\mathfrak{m}$ is
	called \emph{maximal} iff it is equal to any other proper ideal containing it iff --since $R$ is commutative-- the factor algebra 
	$R/\mathfrak{m}$ is a field. Recall that a
	\emph{multiplicative set} $S\subset R$ is a subset $S$
	of $R$
	containing $1$ and if $s,s'\in S$ then $ss'\in S$. Moreover a general proper ideal 
	$\mathfrak{p}$ of
	$R$ is called
	a \emph{prime ideal} iff the factor algebra 
	$R/\mathfrak{p}$ is a domain
	iff the complementary set $S=R\setminus \mathfrak{p}$ is a \emph{multiplicative subset}.
	Recall that \emph{Krull's Lemma} states that given any multiplicative subset $S\subset R$ and ideal $J$ with
	$J\cap S=\emptyset$ there is a prime ideal $\mathfrak{p}\supset J$
	with $S\cap \mathfrak{p}=\emptyset$, see e.g.~\cite[p.391, Prop.~7.2, Prop.~7.3]{Jac80}.

	\subsubsection
	{Analytical features of smooth function algebras}
	\label{SubSubSec Commutative algebra of smooth function algebras}
	
	Let $X$ be an $N$-dimensional differentiable manifold (whose underlying topological space
	we shall always assume to be Hausdorff and second countable).
	Let $\mathbb{K}$ denote either the field of all real numbers, 
	$\mathbb{R}$, or the field of all complex numbers,  $\mathbb{C}$.
	For any real vector bundle $E$ over $X$ we shall denote by the same symbol
	$E$ its complexification. Consider the $\mathbb{K}$-algebra 
	$A=\mathcal{C}^\infty(X,\mathbb{K})$ of
	all smooth $\mathbb{K}$-valued functions $f$ on $X$. Even in the case
	where $X$ is an open subset of $\mathbb{R}^n$ the algebraic properties
	of $A$ are rather different from the function algebras used in
	algebraic or analytic geometry. There has been much work on that
	in the past, see e. g. \cite{Mal67}, \cite{Tou72}, based on the classical works
	by Whitney. We shall give a short outline of the features we shall need.
	
	The $\mathbb{K}$-vector space $A$ is given a well-known Fr\'{e}chet
	topology which can be conveniently defined in the following terms:
	fix a Riemannian metric $h$ on $X$, and let $\nabla$ denote its
	Levi-Civita connection. For any nonnegative integer $n$ denote by
	$\mathsf{S}^nT^*X$ the $n$th symmetric power of the cotangent bundle
	of $X$: its smooth sections can be viewed as smooth functions
	on the tangent bundle $\tau_X:TX\to X$ which are of homogeneous polynomial degree $n$
	in the direction of the fibres. For any section $\alpha\in 
	\Gamma^\infty(X,\mathsf{S}^nT^*X)$ let $D\alpha\in\Gamma^\infty(X,\mathsf{S}^{n+1}T^*X)$ be it is symmetrized
	covariant derivative w.r.t.~$\nabla$ which can be seen as a
	symmetric version (depending on $\nabla$!) of the exterior derivative.
	Finally for any smooth function $f:X\to\mathbb{K}$ let 
	$D^nf\in\Gamma^\infty(X,\mathsf{S}^nT^*X)$ be the $n$fold iterated
	symmetrized covariant derivative of $f$, hence $D^0f=f$, $Df=df$,
	$D^{n+1}f=D(D^nf)$. For any compact set $K\subset X$ and any nonnegative
	integer $m$ define a system of functions $p_{K,m}:A\to \mathbb{R}$ by
	\begin{equation}\label{EqDefSeminorms}
		p_{K,m}(f)=\max\{|D^nf(v)|~|~n\leq m,~\tau_X(v)\in K~
		\mathrm{and}~h(v,v)\leq 1\},
	\end{equation}
	which will define an exhaustive system of seminorms, hence a locally
	convex topological vector space which is known to be metric and
	sequentially complete, hence Fr\'{e}chet.
	It is not hard to see that the choice of another Riemannian metric
	will give another system of seminorms which is equivalent to the first one. For flat $\mathbb{R}^n$ equipped with the usual euclidean scalar product these seminorms are easily seen to be equivalent to the
	usual seminorms used in analysis where the higher partial derivatives
	are expressed by multi-indices. Pointwise multiplication and
	evaluation at a point are well-known to be continuous w.r.t.~to the Fr\'{e}chet topology.\\
	For later use we shall give the usual definition of \emph{multidifferential operators}
	$D:A\times \cdots\times A\to A$ of rank $p$ as a $p$-linear
	map (over the ground field $\mathbb{K}$) such that there is a nonnegative integer $l$ and for each
	chart $(U,(x^1 ,\ldots,x^N))$ there are smooth functions
	$D^{\alpha_1\cdots\alpha_p}:U\to \mathbb{K}$ indexed by
	$p$ multi-indices $\alpha_1,\ldots,\alpha_p\in\mathbb{N}^N$ such that for each point $x\in U$, and smooth functions
	$f_1,\ldots,f_p\in A$ there is the following local expression
	\begin{equation}\label{EqDefMultiDifferentialOperatorsAn}
		D(f_1,\ldots,f_p)(x)
		=\sum_{|\alpha_1|,\ldots,|\alpha_p|\leq l}
		D^{\alpha_1\cdots\alpha_p}(x)
		\frac{\partial^{|\alpha_1|}\big(f_1|_U\big)}
		{\partial x^{\alpha_1}}(x)\cdots
		\frac{\partial^{|\alpha_p|}\big(f_p|_U\big)}
		{\partial x^{\alpha_p}}(x)     
	\end{equation}
	where as usual $|\alpha|=|(i_1,\ldots,i_N)|=i_1+\ldots+i_N$ and $\frac{\partial^{|\alpha|}\phi}{\partial x^\alpha}$ is short for $\frac{\partial^{i_1+\ldots+i_N}\phi}
	{\partial (x^1)^{i_1}\cdots \partial (x^N)^{i_N}}$.
	Note that the value of $D(f_1,\ldots,f_p)$ at $x$ only
	depends on the restriction of the functions $f_1,\ldots, f_p$
	to any open neighbourhood of $x$: it follows that multidifferential operators can always be \emph{localized in the analytical
		sense} that they give rise to unique well-defined multidifferential
	operators $D_U$ on $\mathcal{C}^\infty(U,\mathbb{K})$ for any
	open subset $U\subset X$ such that $D$ and $D_U$ intertwine the
	restriction map $\eta:A\to \mathcal{C}^\infty(U,\mathbb{K})$
	in the obvious way.
	Multidifferential operators
	are well-known to be continuous w.r.t.~to the Fr\'{e}chet topology. Furthermore, recall the usual composition rule
	of multidifferential operators (inherited by the usual rule for multi-linear maps): given two multidifferential
	operators $D$ (of rank $p$) and
	$D'$ (of rank $q$) and a positive integer $1\leq i\leq p$ then the map $D\circ_i D'$ defined by $(f_1,\ldots,f_{p+q-1})\mapsto 
	D\big(f_1,\ldots,f_{i-1},D'(f_i,\ldots,f_{i+p-1}),
	f_{i+p},\ldots,f_{p+q-1}\big)$ is a multidifferential operator of rank $p+q-1$. This composition rule
	obviously is compatible with localization in the sense
	that $\big(D\circ_i D'\big)_U=D_U\circ_i D'_U$.\\
	Consider for any given point
	$x_0\in X$ the binary relation $\sim_{x_0}$ on $A$ defined by $f\sim_{x_0} g$ iff the two smooth functions $f$ and $g$ have the same Taylor series at $x_0$
	w.r.t.~to some chosen chart around $x_0$. It is 
	well-known to be an equivalence relation which does not depend on the chosen chart, an equivalence class is called an \emph{infinite jet}, and
	the class of $f$ is called \emph{the infinite jet $j^\infty_{x_0}(f)$ of $f$}, see e.g.~\cite[p.117 section 12]{KMS93}.\\
	Apart from the
	trivial case where $X$ is a point, $A$ is well-known to have two `bad' features from the point of view of commutative algebra: firstly the product of two non-zero
	functions with disjoint supports --which exist in $A$-- clearly vanishes showing that $A$ \emph{has very many nontrivial 
	zero-divisors}. Secondly $A$ \emph{is NOT Noetherian}: if for each nonnegative integer $n$ we denote by $I_n$ the ideal of all smooth $\mathbb{K}$-valued functions on $\mathbb{R}$ vanishing
	on the closed interval 
	$\left[-\frac{1}{n+1},\frac{1}{n+1}\right]$, then the
	ascending sequence $I_n\subset I_{n+1}$ never stabilizes
	after a finite number of steps.
	
	\subsubsection{Some ideal theory of smooth function algebras}
	
	In this paragraph we collect some facts of maximal
	and prime ideals of $A=\mathcal{C}^\infty(X,\mathbb{K})$
	which are major topics in commutative algebra.
	The main source will be J.-C.~Tougeron's classic
	\cite{Tou72}. This paragraph can be ignored by the impatient reader.\\
	For a given ideal $J$ of $A$ its zero set $Z(J)\subset X$ is of course a closed subset of $X$, and the closure $\overline{J}$ of $J$ (w.r.t.~the Fr\'{e}chet topology) in $A$ remains an ideal of $A$. The vanishing ideal of any set is closed in the Fr\'{e}chet topology, hence there is the chain of inclusions
	$J\subset \overline{J}\subset I(Z(J))$. Since for any point $x$ not contained in $Z(J)$ there is a function $g$ in the ideal $J$ not vanishing at $x$, a simple partition of unity argument shows that any smooth $\mathbb{K}$-valued function whose support is compact and has empty intersection with $Z(J)$ must be an element of $J$. It follows in particular that an ideal $J$ contains the ideal $\mathcal{D}(X)$ of all smooth $\mathbb{K}$-valued functions
	having compact support iff $Z(J)=\emptyset$ iff $J$ is dense in $A$. In the particular case where $X$ is compact this means that the only dense ideal is equal to $A$.
	Returning to general $X$ it follows that
	every proper ideal of $A$ which is \emph{closed} w.r.t.~the
	Fr\'{e}chet topology has a non-empty set of common zeros.\\
	In particular, every \emph{closed maximal ideal} of $A$
	is equal to the vanishing ideal $I_{x_0}=I(\{x_0\})$ of some
	point $x_0\in X$.\\
	In general, for the closure of an ideal there is the very useful \emph{Whitney's Spectral Theorem} stating that
	a function $g$ belongs to the closure $\overline{J}$ of an
	ideal $J$ of $A$ iff for each $x\in X$ there is a function $h\in J$
	(whose choice may depend on $x$)
	whose infinite jet $j_x^\infty(h)$ is equal to $j_x^\infty(g)$, see e.g.~\cite[p.~91, Cor.~1.6., cas $q=1$]{Tou72}. Moreover, ideals having finitely many analytic generators
	are always closed, see e.g.~\cite[p.119, Cor.~1.6.]{Tou72},
	but there are also closed ideals having finitely many nonanalytic generators, see e.g.~\cite[p.104,~Rem.~4.7, Exemp.~4.8.]{Tou72}. \\
	In the following, given $x_0\in X$, denote by $\mathfrak{I}_{x_0}$ the ideal of $A$ consisting of all smooth functions
	vanishing in some neighbourhood of $x_0$, and by
	$I^\infty_{x_0}$ the ideal of $A$ consisting of all functions $f$ such that $j^\infty_{x_0}(f)=0$.
	Clearly $\mathfrak{I}_{x_0}\subset I^\infty_{x_0}\subset I_{x_0}$.
	Consider now a \emph{prime ideal} 
	$\mathfrak{p}\subset A$  of $A$. We know that it is either dense iff $Z(\mathfrak{p})=\emptyset$ or has a
	nonempty zero set. For each prime ideal it can be shown
	that
	\begin{equation}\label{EqCompProperPrimeFirstProperties}
	 Z(\mathfrak{p})\neq \emptyset
	 ~~\Leftrightarrow~~\exists ~ x_\mathfrak{p}\in X:
	  Z(\mathfrak{p}) = \{x_\mathfrak{p}\}~~\Leftrightarrow~~
	  \exists ~ x_0\in X:~\mathfrak{I}_{x_0}\subset \mathfrak{p}~~\Leftrightarrow~~
	     \exists ~ y_0\in X: \mathfrak{p}\subset I_{y_0}.
	\end{equation}
	and in case one the four equivalent statements is fulfilled then $x_0=y_0=x_\mathfrak{p}$, uniquely determined by $\mathfrak{p}$. \\
	Indeed, it is obvious that in
	eqn (\ref{EqCompProperPrimeFirstProperties}) the second statement implies the first which is equivalent to the fourth. Moreover, if $Z(\mathfrak{p})$ contained
	two distinct points $x_1,x_2\in X$ there would be two smooth functions
	$\varphi_1,\varphi_2\in A$ with disjoint supports such that $\varphi_1(x_1)=1=\varphi_2(x_2)$ (whence 
	$\varphi_1,\varphi_2\in A\setminus \mathfrak{p}$),
	but $\varphi_1\varphi_2=0\in\mathfrak{p}$ contradicting the fact that 
	$\mathfrak{p}$ is prime whence the first, the second,
	and the last statement of (\ref{EqCompProperPrimeFirstProperties})
	are equivalent implying the uniqueness and equality of
	$x_\mathfrak{p}$ and $y_0$ in case one of three statements
	is fulfilled. Moreover supposing that $Z(\mathfrak{p})=\{x_\mathfrak{p}\}$ then for any $h\in\mathfrak{I}_{x_\mathfrak{p}}$ there is $\varphi\in A$ with 
	$\varphi(x_\mathfrak{p})=1$ having its support inside the open neighbourhood of $x_\mathfrak{p}$ on which $h$ vanishes, whence
	$\varphi\in A\setminus \mathfrak{p}$, but 
	$\varphi h=0\in\mathfrak{p}$ so $h\in\mathfrak{p}$ since
	$\mathfrak{p}$ is prime, whence the second statement
	of (\ref{EqCompProperPrimeFirstProperties}) implies the
	third.  Finally, supposing $\mathfrak{I}_{x_0}\subset \mathfrak{p}$ for some $x_0\in X$, if there was a $g\in\mathfrak{p}$
	with $g(x_0)\neq 0$ one would find a positive valued
	function $h\in \mathfrak{I}_{x_0}\subset \mathfrak{p}$
	and a bump function $\chi\in A$ such that 
	$\chi|g|^2+h\in\mathfrak{p}$ has only strictly positive values, hence is invertible
	implying $\mathfrak{p}=A$  which contradicts the fact that $\mathfrak{p}$ is proper. Hence $\mathfrak{p}\subset I_{x_0}$ whence
	the third statement implies the last
	in eqn (\ref{EqCompProperPrimeFirstProperties}) and the
	equality $x_\mathfrak{p}=x_0=y_0$.
	\\
	Next we shall look at \emph{closed prime ideals}
	of $A$:
	fix a point $x_0\in X$, then
	by Borel's classical Lemma  (see e.g.~\cite[p.332,~Satz~5.3.33]{Wal07})
	the factor algebra $A/I^\infty_{x_0}$ is isomorphic
	to the algebra of formal power series 
	$\mathbb{K}[[x_1,\ldots,x_n]]$ which is a domain whence
	$I^\infty_{x_0}$ is a prime ideal which is closed
	since $f\mapsto j^k_{x_0}(f)$ is continuous for every
	nonnegative integer $k$. Moreover the obvious inclusion
	$\mathfrak{I}_{x_0}\subset I^\infty_{x_0}$ implies
	$\overline{\mathfrak{I}_{x_0}}\subset I^\infty_{x_0}$ since $I^\infty_{x_0}$ is closed. 
		Since for any
	$g\in I^\infty_{x_0}$ we have by definition 
	$j^\infty_{x_0}(g)=0=j^\infty_{x_0}(0)$, and for any
	$y\in X\setminus\{x_0\}$ there is a smooth function
	$\chi\in A$ vanishing in a suitable open neighbourhood of
	$x_0$ and having the constant value $1$ in another suitable open neigbourhood of $y$ it follows that
	$\chi g\in \mathfrak{I}_{x_0}$ and 
	$j^\infty_{y}(\chi g)=j^\infty_{y}(g)$ whence
	$g\in \overline{\mathfrak{I}_{x_0}}$ thanks to Whitney's
	spectral theorem. This implies the equality 
	$\overline{\mathfrak{I}_{x_0}}= I^\infty_{x_0}$.
	By passing to closures in eqn
	(\ref{EqCompProperPrimeFirstProperties}) it immediately follows that for any proper closed prime ideal $\mathfrak{p}$
	\begin{equation}
	 \label{EqCompPropertiesClosedPrimes}
	\mathrm{if}~\mathfrak{p}=\overline{\mathfrak{p}}\neq A
	~\mathrm{and}~Z(\mathfrak{p})=\{x_\mathfrak{p}\}:
	~~~~~\mathfrak{I}_{x_\mathfrak{p}}\subset \overline{\mathfrak{I}_{x_\mathfrak{p}}} 
	= I^\infty_{x_\mathfrak{p}}
	  \subset \mathfrak{p} \subset I_{x_\mathfrak{p}}.
	\end{equation}
 Conversely, another simple application of Whitney's spectral theorem shows that for any prime ideal $\mathfrak{p}$ the inclusion
 $I^\infty_{x_0}\subset \mathfrak{p}\subset I_{x_0}$ implies
 that $\mathfrak{p}$ is proper and closed.
 Moreover, it follows that for each given $x_0\in X$ the set of all closed prime ideals $\mathfrak{p}\subset A$ with 
 $Z(\mathfrak{p})=\{x_0\}$ is in bijection with
 the set of all the prime ideals of the formal power series
 algebra $\mathbb{K}[[x_1,\ldots,x_n]]$ via the map
 $\mathfrak{p}\mapsto \mathfrak{p}/I_{x_0}^\infty$.
 These latter prime ideals can be characterized in a purely algebraic way, see e.g.~ 
 \cite[p.~31,~Prop.~2.2]{Tou72}. Thirdly it is somewhat harder to see that
 $I^\infty_{x_0}I^\infty_{x_0}=I^\infty_{x_0}$, see
\cite[p.93, Lemme 2.4]{Tou72}
(for the particular case where the closed set equals 
 $\{x_0\}$) implying that $I^\infty_{x_0}$ is equal to the intersection of all the powers of $I_{x_0}$.\\
 Note that there are very many `funny' non closed prime
 ideals of $A$ (even in the case where $X$ is compact):
 by applying Krull's Lemma to the ideal $\mathcal{D}(X)$
 and the multiplicative subset generated by an arbitrary fixed function $f$ having non compact support we have the existence of a dense prime ideal which does not contain $f$. Likewise, applying Krull's Lemma to the
 ideal $\mathfrak{I}_{x_0}$ and the multiplicative subset
 generated by an arbitrary function $g\in I^\infty_{x_0}$
 which is not in $\mathfrak{I}_{x_0}$ we get a proper
 non closed prime ideal $\mathfrak{p}$ with 
 $Z(\mathfrak{p})=\{x_0\}$, hence containing $\mathfrak{I}_{x_0}$, but not $I^\infty_{x_0}$.

	\subsection{(Algebraic) Localization}
	\label{SubSecAlgebraicLocalization}
	
	This section recalls well-known results which we present
	according to the excellent text-book \cite{Lam99} in a categorically
	`tuned' version. See also the rather useful review
	\cite{Sko2006} for more aspects.

	\subsubsection{Commutative Localization}
	
	Recall that for any domain $ R $ it is always possible to construct a field, called the field of fractions of  $ R $, by formally inverting all nonzero elements. 
	More generally,  recall the \emph{localization of a commutative 
		$K$-algebra $R$}: let $S\subset R$ be a \emph{multiplicative subset}
	(which is characterized by containing the unit and for any two
	of its elements its product). Then the following binary relation
	$\sim$ on $R\times S$ defined by
	\begin{equation}\label{EqDefFractionsEquivalenceRelComm}
		(r_1,s_1)\sim (r_2,s_2)~~~\mathrm{if~and~only~if}~~~
		\exists~s\in S:~r_1s_2s=r_2s_1s
	\end{equation}
	is an equivalence relation, and the set of all classes (written as (symbolic) fractions $\frac{r}{s}$) forms a commutative
	$K$-algebra $R_S$ --by means of the usual addition and multiplication rules of fractions-- called the \emph{quotient ring}, and a ring homomorphism (the \emph{numerator morphism}) 
	$\eta_{(R,S)}=\eta:R\to R_S$ given by $r\mapsto \frac{r}{1}$
	which in particular defines the $K$-algebra structure of 
	$R_S$.
	Let $U(R)\subset R$ denote the multiplicative group of invertible
	elements of $R$. A morphism of unital $K$-algebras
	$\Phi:R\to R'$ is called \emph{$S$-inverting} (for a multiplicative
	subset $S\subset R$) if for each $s\in S$ the image
	$\Phi(s)$ is invertible in $R'$, hence $\Phi(S)\subset U(R')$. The following properties of the constructions can be observed:
	\begin{prop} \label{PLocalizationCommutative}
		Let $R$ be a commutative $K$-algebra and 
		$S\subset R$ be a multiplicative subset. Then the following
		is true:
		\begin{enumerate}
			\item[a.] $ \eta_{(R,S)}(S)\subset U(R_S)$, that is, the homomorphism $\eta_{(R,S)}$ sends elements of  $S$ to invertible elements of $R_S$. Moreover, for any commutative unital $K$-algebra $R$ equipped with a multiplicative subset $S\subset R$, the pair $(R_S,\eta_{(R,S)})$ is \textbf{universal} in the sense that any $S$-inverting morphism of unital  
			$K$-algebras
			uniquely factorizes, i.e.~the following diagram commutes:
			\begin{equation}\label{EqUniversalityOfNumeratorMap}
				\xymatrix{ R \ar[r]^{\eta}\ar[dr]_\alpha &  R_S  \ar[d]^f \\
					&	R' & }
			\end{equation}
			where $f$ is a morphism of unital $K$-algebras determined by
			$\alpha$, see e.g.~\cite[p.55, Ch.III]{Mac98} 
			for definitions of universal objects.
			\item[b.] Every element of $ R_S $ can be written as
			a fraction 
			$\eta(r)\eta(s)^{-1}, $ for some $r\in R $ and $s\in S$.
			\item[c.] $ \ker(\eta_{(R,S)})=
			\{r\in R~|~rs=0\mathrm{\:for\:some\:}s\in S\} $.
		\end{enumerate}
	\end{prop}
	\noindent We shall give a more categorical description in the next section.\\
	
	\noindent \textbf{Remarks} (which will only be used in Section
	\ref{SecCommutativelyLocalizedStarProducts}):
	\begin{enumerate}
		\item Recall the well-known \emph{localization of any $R$-module $M$}, see
		e.g.~\cite[p.397]{Jac80}, which is a module $M_S$ with respect to the localized algebra $R_S$. It is naturally
		isomorphic to $R_S\otimes_R M$, see e.g.~\cite[p.398, Prop.7.6]{Jac80}.
		\item Let $R,R'$ be commutative associative unital
		$K$-algebras, and $S\subset R$, $S'\subset R'$ multiplicative subsets, respectively.
		Then it is straight-forward to see that $S\otimes_K S'=\{s\otimes_K s'\in 
		R\otimes_K R'~|~s\in S; s'\in S'\}$ is a multiplicative subset of the $K$-algebra $R\otimes_K R'$ and that the tensor product of the numerator morphisms
		$\eta_{(R,S)}\otimes_K\eta_{(R',S')}:R\otimes_K R'\to
		R_S\otimes_K R'_{S'}$ induces a natural isomorphism
		of unital $K$-algebras
		\begin{equation}
		\label{EqCompLocalizationOfTensorProducts}
		  (R\otimes_K R')_{S\otimes_K S'} ~\stackrel{\sim}{\longrightarrow}~
		  R_S\otimes_K R'_{S'}.
		\end{equation}
	\end{enumerate}

	\subsubsection{Noncommutative Localization: General Construction}
	
	Let $R$ be an associative unital $K$-algebra which is not necessarily commutative. Again, we call 
	$S\subset R$ a multiplicative subset if for all $s,s'\in S$ we have $ss'\in S$ and $1_R=1\in S$. As above, let $U(R)\subset R$ denote the multiplicative subset (which is even a group) of invertible elements of $R$.
	\\
	Let $K \mathbf{Alg}$ be the category of all associative unital
	$K$-algebras. 
	Moreover, let $K \mathbf{AlgMS}$ be the category of all pairs $(R,S)$ of associative unital $K$-algebras $R$
	with a multiplicative subset $S\subset R$ where the morphisms
	$(R,S)\to (R',S')$ are morphisms of unital $K$-algebras
	$R\to R'$ mapping $S$ into $S'$. Since any morphism of unital $K$-algebras maps the group of invertible elements in the group of invertible elements there is an obvious functor $\mathcal{U}:K \mathbf{Alg}\to K \mathbf{AlgMS}$ given on
	objects by $\mathcal{U}(R)=\big(R,U(R)\big)$.\\
	For \emph{commutative $K$-algebras}, the above localization description in Proposition \ref{PLocalizationCommutative}, 
	$a.$, gives rise to a functor 
	$\mathcal{L}:K \mathbf{AlgMS}\to K \mathbf{Alg}$ associating
	to each pair $(R,S)$ the quotient ring $R_S$, and it is not hard to see that it is a \emph{left adjoint} of the functor $\mathcal{U}$, see e.g.~\cite[p.79, Ch.IV]{Mac98} for definitions: the unit of the adjunction gives back the
	canonical numerator morphism $\eta$, and the counit is an isomorphism since localization w.r.t.~the group of all invertible elements is isomorphic to the original algebra.
	
	In the general noncommutative situation such a localization
	functor $\mathcal{L}:K \mathbf{AlgMS}\to K \mathbf{Alg}$
	does also always exist, see e.g.~\cite[Prop.(9.2), p.289]{Lam99} for a proof. We present it in the following categorical form:
	\begin{prop}\label{PGeneralLocalization}
		There is an adjunction of functors
		\[
		K\mathbf{AlgMS}~~\begin{array}[c]{c}
			\underrightarrow{~~~~~\mathcal{L}~~~~~} \\
			\overleftarrow{~~~~~\mathcal{U}~~~~~}
		\end{array}~~
		K\mathbf{Alg}
		\]
		where $\mathcal{L}$ is the left adjoint to the above functor
		$\mathcal{U}$
		such that each component $\eta_{(R,S)}$ of the unit 
		$\eta:I_{K \mathbf{AlgMS}}\naturalto \mathcal{U}\mathcal{L}$ 
		of the adjunction
		satisfies the universal property $a.$ of
		the previous Proposition \ref{PLocalizationCommutative}
		in the general noncommutative case.
		We refer to $\mathcal{L}$ as a \textbf{localization functor}.\\
		For a given $(R,S)$ in $K\mathbf{AlgMS}$ we denote
		by $R_S$ the $K$-algebra $\mathcal{L}(R,S)$ given by 
		the functor $\mathcal{L}$, and by 
		$\eta_{(R,S)}:R\to R_S$ the component of the unit of the adjunction. Then $\eta_{(R,U(R))}:R\to R_{U(R)}$ is an
		isomorphism, the inverse being the component
		$\epsilon_R$ of the counit 
		$\epsilon:\mathcal{L}\mathcal{U}\naturalto
		I_{K\mathbf{Alg}}$ of the adjunction.
		Moreover, every element of the $K$-algebra $R_S$ is a finite
		sum of products of the form ($\eta=\eta_{(R,S)}$)
		\begin{equation}
		\label{EqCompLocAlgebraGeneralElement}
			\eta(r_1)\big(\eta(s_1)\big)^{-1}\cdots 
			\eta(r_N)\big(\eta(s_N)\big)^{-1}
		\end{equation}
		(which may be called `multifractions') with $r_1,\ldots,r_N\in R$ and $s_1,\ldots,s_N\in S$
		(note that $r_1$ or $s_N$ may be equal to the unit element of $R$).
	\end{prop}
	The idea of the proof of \cite[Prop.(9.2), p.289]{Lam99} is as follows: (see also the PhD thesis 
	\cite[p.144]{Ara21} for details) there is a natural
	surjective morphism of unital $K$-algebras $\hat{\epsilon}_R$ from the
	free $K$-algebra generated by the $K$-module $R$,
	$T_KR$, to $R$ which provides us with a natural categorical presentation of 
	$R$ `by generators and relations': this morphism is given by the $R$-component of the counit $\hat{\epsilon}$ of the well-known adjunction 
	\[
	K\mathbf{Mod}~~\begin{array}[c]{c}
		\underrightarrow{~~~~~
		{T_K}~~~~~} \\
		\overleftarrow{~~~~~\mathcal{O}~~~~~}
	\end{array}~~
	K\mathbf{Alg}
	\] 
	where $\mathcal{O}$ is the forgetful functor and ${T_K}$ the free algebra functor. Let
	$\kappa(R)\subset T_KR$ denote the kernel of 
	$\hat{\epsilon}_R$. The next step is to add to the generating $K$-module $R$ the free $K$-module
	$KS$ with basis $S$, and to consider the two-sided ideal $\kappa(R,S)$ in
	the free algebra $T_K(R\oplus KS)$ generated by 
	$\kappa(R)$ and by the subsets 
	$\{(s,0)\otimes(0,s)-\mathbf{1}_T~|~s\in S\}$
	and $\{(0,s)\otimes (s,0)-\mathbf{1}_T~|~s\in S\}$
	of $T_K\big(R\oplus KS\big)$ where the multiplication $\otimes$ and the unit $\mathbf{1}_T$ are taken in the free
	algebra $T_K\big(R\oplus KS\big)$. The localized algebra
	$\mathcal{L}(R,S)=R_S$ is then defined by 
	$R_S=T_K\big(R\oplus KS\big)/\kappa(R,S)$, and the
	`numerator morphism' $\eta_{(R,S)}:R\to R_S$ is simply 
	the canonical injection of $R$ into $T_KR\subset T_K\big(R\oplus KS\big)$ followed by the obvious projection. It follows that for every
	$s\in S$ its image $\eta_{(R,S)}(s)$ has an inverse by
	construction. The verification that this leads to a well-defined functor $\mathcal{L}$ which is a left adjoint to
	the functor $\mathcal{U}$ is lengthy, but straight-forward. 
	
	The preceding construction shows that the
	functor $\mathcal{L}$ provides us with an abstract universal numerator map $\eta_{(R,S)}$ which is
	\emph{$S$-inverting} in the sense that every $\eta_{(R,S)}(s)$,
	$s\in S$, is invertible in $R_S$ and a natural isomorphism $\epsilon_R$
	from an algebra to its localization w.r.t.~its group of units.
	
\subsubsection{Noncommutative Localization: Ore Localization}

Although the preceding general localization construction
is always well-defined, it exhibits the following draw-backs
which show the need for a more particular localization procedure due to {\O}.~ Ore, 1931, \cite{Ore31} which we shall sketch in this Section:

\begin{itemize}	
\item The construction by generators and relations renders
	the localized algebra $R_S$ quite implicit and not always
	computable. 
\item Of course, it is easy to see that if $S$ contains $0$ then
	the localized algebra is trivial,
	$R_S\cong\{0\}$. But even for multiplicative subsets $S\subset R$
	not containing $0$ it may happen that the localized algebra
	$R_S$ is trivial as example $(9.3)$ of 
	\cite[p.289]{Lam99} shows.
	This can never happen in the commutative case since
	the equation $\frac{1}{1}=\frac{0}{1}$ is equivalent to
	the fact that $0\in S$. This shows the lack of control
	over the kernel of the `numerator morphism' $\eta_{(R,S)}$.

\item The presentation of elements of $R_S$ in terms of 
	sums of `\emph{multifractions}'
	as equation (\ref{EqCompLocAlgebraGeneralElement}) shows
	is quite clumsy, and one would prefer simple right or
	left fractions.
\end{itemize}
	
	In order to motivate the particular conditions on $S$ in the following definition we look at the multifractions which span
	the localized $K$-algebra $R_S$, see eqn (\ref{EqCompLocAlgebraGeneralElement}): it may be desirable to transform a multifraction in a simple right fraction, and a partial step may consist in transforming a left fraction
	$\big(\eta(s)\big)^{-1}\eta(r)$ (with $r\in R$ and $s\in S$) directly into a right fraction $\eta(r')
	\big(\eta(s')\big)^{-1}$ (for some $r'\in R$ and $s'\in S$) which implies that every multifraction is equal to a right fraction by applying this step a finite number of times. This above condition implies the equation $\eta(rs')=\eta(sr')$ and thus motivates the stronger condition that for any pair $(r,s)\in R\times S$ there is a pair $(r',s')\in R\times S$ such that $rs'=sr'$, and this is the well-known \emph{right Ore condition}:

	%
	%
	
	\begin{defi}\label{defi: bigdef}
		Let $R$ be an associative unital $K$-algebra, 
		and $S\subset R$ be a multiplicative subset.	
		\begin{itemize}
			\item[i.]A $K$-algebra $\check{R}_S$ equipped with a 
			morphism of  unital $K$-algebras $\check{\eta}_{(R,S)}=\check{\eta}:R\to \check{R}_S$ is said to be a \textbf{right $K$-algebra of fractions of $(R,S)$} if the following
			conditions are satisfied:
			\begin{itemize}
				\item[a.] $\check{\eta}_{(R,S)}$ is $S$-inverting,
				\item[b.] Every element of  $\check{R}_S$ is of the form 
				$\check{\eta}(r)\big(\check{\eta}(s)\big)^{-1} $ for some 
				$r\in R $ and $s\in S$;
				\item[c.] $ \ker(\check{\eta})
				=\{r\in R~|~rs=0, \mathrm{\:for\:some\:}s\in S\}=:I_{(R,S)}=:I$. 
			\end{itemize}
			\item[ii.] $S$ is called a \textbf{right denominator set}
			if it satisfies the following two properties:
			\begin{itemize}
				\item[a.] For all $r\in R$ and $s\in S$ we have 
				$rS\cap sR\neq\emptyset$ ($S$ \textbf{right permutable} or \textbf{right Ore set}), i.e.~there are $r'\in R$
				and $s'\in S$ such that $rs'=sr'$.
				\item[b.] For all $r\in R$ and for all $s'\in S$: if $s'r=0$
				then there is $s\in S$ such that $rs=0$ 
				($S$ \textbf{right reversible}). 
			\end{itemize}
		\end{itemize}
	\end{defi}
	\noindent In case $R$ is commutative every multiplicative subset is
	a right denominator set. Moreover the group of all invertible
	elements $U(R)$ of any unital $K$-algebra is obviously
	a right denominator set.
	
	\noindent The next theorem shows that such a right algebra of fractions
	exists iff $S$ is a right denominator set, see also
	\cite[Thm (10.6), p.300]{Lam99}:

	\begin{theorem} \label{TLocalizationForRightDenominatorSets}
		Let $R$ be a unital $K$-algebra and $S\subset R$ be
		a multiplicative subset. Then the following is true:
		\begin{enumerate}
			\item 
			The $K$-algebra $R$ has a right $K$-algebra of fractions $\check{R}_S$ with respect to the multiplicative subset $S$ if and only if $S$ is a right denominator set. 
			\item If this is the case each such pair $(\check{R}_S,\check{\eta})$
			is universal in the sense of diagram (\ref{EqUniversalityOfNumeratorMap}) and each $\check{R}_S$ is isomorphic
			to the canonical localized algebra $R_S$ of
			Proposition \ref{PGeneralLocalization}.
			\item Each $\check{R}_S$ is isomorphic to the quotient set 
			$RS^{-1}:=(R\times S)/\sim$ with respect to
			the following binary relation $\sim$ on $R\times S$
			\begin{equation}\label{EqDefRightFractionsEquivalenceRel}
				(r_1,s_1)\sim(r_2,s_2) ~~ \Leftrightarrow~~\exists b_1,b_2\in R~
				\mathrm{such~that}~s_1b_1=s_2b_2\in S\mathrm{\:and\:} r_1b_1=r_2b_2 \in R 
			\end{equation}
			which is an equivalence relation generalizing relation (\ref{EqDefFractionsEquivalenceRelComm}). $RS^{-1}$ carries a canonical unital
			$K$-algebra structure, i.e. addition and
			multiplication on equivalence classes
			$r_1s_1^{-1}$ and $r_2s_2^{-1}$ (with $r_1,r_2\in R$ and
			$s_1,s_2\in S$) is given by
			\begin{equation}
				\label{EqAdditionMultiplicationOnEqClassesRD}
				r_1s_1^{-1}+r_2s_2^{-1}=(r_1c_1+r_2c_2)s^{-1},~~~\mathrm{and}~~~(r_1s_1^{-1})(r_2s_2^{-1})
				=(r_1r')(s_2s')^{-1}
			\end{equation}
			where we have written $s_1c_1=s_2c_2=s\in S$ (with $c_1\in S$ and $c_2\in R$) and $r_2s'=s_1r'$ (with
			$s'\in S$ and $r'\in R$) using the right Ore property.
			The numerator morphism $\eta_I:R\to RS^{-1}$ is
			given by $\eta_I(r)=r1^{-1}$ for all $r\in R$.
		\end{enumerate}
	\end{theorem}
	\noindent 
	For a proof, see e.g.~\cite[p.244, Thm.~25.3]{Pas91}
	\footnote{We are indebted to A.~Eduque for having pointed out this reference to us.} or the PhD thesis
	\cite[p.146]{Ara21}.\\
	We shortly describe the \textbf{idea of the proof}: whereas
	in part 1. the verification of the implication ``$ (i.)~ \Longrightarrow ~(ii.) $'' in Definition \ref{defi: bigdef} is straight-forward, the
	converse implication ``$ (i.)~ \Longleftarrow ~(ii.) $''
	of Definition \ref{defi: bigdef} is much more involved:
	the traditional `steep and thorny way' (originally set up by {\O}ystein Ore, \cite{Ore31}) consists of
	a concrete construction of the $K$-algebra $RS^{-1}$ upon
	using the above relation (\ref{EqDefRightFractionsEquivalenceRel}) 
	--which reflects the idea of creating `common denominators'-- and defining and verifying the canonical
	$K$-algebra structure (\ref{EqAdditionMultiplicationOnEqClassesRD}) on the quotient set $R\times S/\sim$ by hand
	which is elementary, but extremely tedious 
	(even the fact that the above relation 
	(\ref{EqDefRightFractionsEquivalenceRel}) is transitive requires some work). We refer to Lam's book \cite[p.300-302]{Lam99} for some of the details. \\
	There is a different more elaborate way to prove
	part 1. and the rest of the theorem (see \cite[p.244, Thm.~25.3]{Pas91} and \cite[Remark (10.13), p.302, 
	and footnote 70]{Lam99}): it is instructive to
	look first at the equivalence relations created by an arbitrary
	$S$-inverting morphism of unital $K$-algebras $\alpha:R\to R'$, the classes being defined by the fibres of the map
	$p_\alpha:R\times S\to R'$ given by $p_\alpha(r,s)=
	\alpha(r)\big(\alpha(s)\big)^{-1}$, which is already
	very close to relation (\ref{EqDefRightFractionsEquivalenceRel}): thanks to the
	fact that
	the right fractions $\alpha(r)\big(\alpha(s)\big)^{-1}$
	form a $K$-subalgebra of $R'$ (here the Ore axiom is
	needed) it creates an
	algebra structure on the quotient set isomorphic to
	the aforementioned subalgebra of $R'$ whence there is no need of tedious
	verifications of identities of algebraic structures. The central point then is to construct
	a unital $K$-algebra $R'$ and an $S$-inverting morphism
	$\alpha:R\to R'$ whose kernel is minimal, hence \emph{equal to
		$I_{(R,S)}$} which finally shows that the above algebra
	$RS^{-1}$ exists and does everything it should do.
	For this construction, the following trick is used:
	after `regularizing' $R$ by passing to the factor algebra $\overline{R}=R/I_{(R,S)}$ (where the image multiplicative set $\overline{S}$ does no longer contain right or left divisors of zero)
	one looks at
	the endomorphism algebra of the \textbf{injective hull $E$ of the
		right $\overline{R}$-module $\overline{R}$}. Every left
	multiplication with elements of $\overline{R}$ can nonuniquely be extended to $E$, and
	the extensions of left multiplications with elements
	of $\overline{S}$ turn out to be invertible (here the Ore axiom is needed). $R'$ will 
	then be given by the subalgebra generated by all extensions
	of left multiplications and the inverses of left multiplications with elements of $\overline{S}$ modulo the two-sided
	ideal of all $\overline{R}$-linear maps $E\to E$ vanishing
	on $\overline{R}$: this will resolve the ambiguity of extension, and $\overline{R}$ injects in $R'$, the injection being $\overline{S}$-inverting.

	Moreover, in any noncommutative domain (no nontrivial zero divisors) which is \emph{right Noetherian}
	(i.e.~where every ascending chain of right ideals stabilizes) the subset of nonzero
	elements is always a right denominator set (see \cite[p.304, Cor.~(10.23)]{Lam99} or \cite[p.14, Beisp.~2.3 b)]{BGR73}).
	In particular, this applies to every universal enveloping algebra over a finite-dimensional Lie algebra 
	(over a field $\mathbb{K}$ of characteristic zero) and for the Weyl-algebra generated by
	$\mathbb{K}^{2n}$. On the other hand, for the free algebra 
	$R=T_\mathbb{K}V$ generated by a vector space $V$ of dimension $\geq 2$ over a field $\mathbb{K}$ of characteristic zero
	(which is well-known to be isomorphic to the
	universal enveloping algebra of the free Lie algebra generated
	by $V$) the multiplicative subset of all nonzero elements is neither
	a right nor a left denominator set: for two linearly independent elements $v$ and $w$ in $V$ we clearly have
	$vR\cap wR=\{0\}$. Hence the above statement about
	universal enveloping algebras does no longer apply to
	infinite-dimensional Lie algebras like the free Lie algebra
	generated by $V$. Moreover inverse images of
	right denominator subsets are in general no right denominator
	subsets as the example of the natural homomorphism
	$T_KV\to S_KV$ of the free to the free commutative algebra generated by $V$ shows: as $S_KV$ is a commutative domain, the subset $S=S_KV\setminus \{0\}$ is a right denominator set whereas its inverse image $T_KV\setminus \{0\}$ is not. On the other hand every homomorphic image of
	a right (or left) Ore set clearly is again a right (or left)
	Ore set. However, there may be subsets of right (or left)
	denominator sets which are no longer right (or left) denominator sets, as we shall see later in Section
	\ref{SecNonOreExample}.

	\subsection{Star products}

	We want to recall some basic definitions and facts about the deformation  quantization of smooth manifolds and star products, see \cite{BFFLS78}, \cite{Wal07} for more information. 
	
	Given a  $\mathbb{K}$-vector space $V$ we denote by $V\ph$ the $\mathbb{K}\ph$-module of formal power series. 
	An element of $v \in V\ph$ can be written uniquely as $v = \sum_{i=0}^\infty v_i \lambda^i$ with $v_i \in V$, and for a 
	given $v\in V\ph$ and $i\in\mathbb{N}$ we shall always write $v_i\in V$ for the $i$th component of $v$ as a formal power series. We also note that for two $\mathbb{K}$-vector spaces $V,W$ we have 
	$\Hom_{\mathbb{K}\ph}(V\ph,W\ph) \cong \Hom_{\mathbb{K}}(V,W)\ph$. \\
	In the following considerations of differential geometry
	we set $\mathbb{K}=\mathbb{R}$ or $\mathbb{K}=\mathbb{C}$,
	and for any smooth differentiable manifold $X$ we
	write $\CCinf(X)=\mathcal{C}^\infty(X,\mathbb{K})$.

	\begin{defi}[Star product]\label{DefStarProducts}
		A (formal)  star product $\star$ on a manifold $X$ is a $\mathbb{K}\ph$-bilinear operation $\CCinf(X)\ph \times \CCinf(X)\ph \rightarrow \CCinf(X)\ph$  --which can always be written as a formal series $f \star g = \sum_{k=0}^\infty \lambda^kC_k(f,g)$ for all $f,g\in\CCinf(X)$--
		satisfying the following properties for all $f,g, \in \CCinf(X)$ (see section \ref{SubSubSec Commutative algebra of smooth function algebras} for definitions and
		notations):
		\begin{itemize}
			\item $\sum_{l=0}^k \big(C_l\circ_1 C_{k-l} 
			-C_l\circ_2 C_{k-l}\big)=0, ~\forall k\geq0. $,
			\item $C_0(f,g) = fg$,
			\item $ 1 \star f = f \star 1  = f$,	
		\end{itemize}
		with $\mathbb{K}$-bilinear operators $C_k: \CCinf(X) \times \CCinf(X) \to \CCinf(X)$ which we
		always assume to be bidifferential operators. 
	\end{defi}
	
	\begin{rem}
		It follows from the first equation of Definition \ref{DefStarProducts} that $\star$ is associative.
	\end{rem}
	\noindent Note that every star product $\star$ can be 
	analytically localized to an associative star product $\star_U$ defined
	on $\mathcal{C}^\infty(U)[[\lambda]]$ by the localization
	of all the bidifferential operators $C_k$ to $C_{kU}$
	(see section \ref{SubSubSec Commutative algebra of smooth function algebras} for more details).
	
	The following well-known explicit star product $\star_s$ on $\mathbb{R}^2$ with coordinates $(x,p)$ will be used
	in the sequel:
	\begin{equation}\label{EqDefStandardStarProduct}
		f\star_s g =\sum_{k=0}^\infty \frac{\lambda^k}{k!}
		\frac{\partial^k f}{\partial p^k}
		\frac{\partial^k g}{\partial x^k}
	\end{equation}
	for any two functions $f,g\in \CCinf(\mathbb{R}^2)$.
	In the physics literature $\lambda$ corresponds to
	$(-\mathbf{i}\hbar)$. Moreover, for functions polynomial
	in the `momenta' $p$ it is obvious that the above series
	converges, and for $\lambda=1$ one obtains the usual formula
	for the symbol calculus of multiplication of differential operators on the
	real line (where partial derivatives are always brought to the right and replaced by the new variable $p$).

	The star commutator for $a,b \in \CCinf(X)\ph$ is defined by 
	$[a,b]_\star = a \star b - b\star a$.
	As usual, the star commutator satisfies the Leibniz-identity, i.e. $[a,b \star c]_\star=[a,b]_\star \star c + b \star [a,c]_\star$, and the Jacobi-identity and thus defines the
	structure of a non-commutative Poisson algebra. Also the adjoint action is a derivation of $\CCinf(X)\ph$ for all $a \in \CCinf(X)\ph$.\\
	From this it can easily be deduced that the first order term of a star product defines a Poisson bracket as follows 
	\begin{equation}
		\{f,g\} =\frac{1}{2} ( C_1(f,g)-C_1(g,f)) = \frac{1}{2\lambda} [f,g] |_{\lambda=0} \text{ for } f,g \in \CCinf(X).
	\end{equation}
	For $\CCinf(X)$ it is well-known that every Poisson bracket comes from a unique \emph{Poisson structure} $\pi$ which is a smooth
	bivector field $\pi$, i.e.~a smooth section in 
	$\Lambda^2TX$ satisfying the identity $[\pi,\pi]_S=0$
	where $[~,~]_S$ denotes the Schouten bracket, see 
	e.g.~\cite[p.84-87]{Wal07}: the relation is
	$\{f,g\}=\pi(df,dg)$.
	The very difficult converse problem	whether the Poisson bracket associated to any given Poisson structure  $\pi$
	arises as the first order commutator of a star product
	had been solved by M.~Kontsevich, see \cite{Kon03}.\\
	The following considerations will only be used in
	Section \ref{SecNonOreExample}: two star products $\star$, $\star'$ are called \emph{equivalent} if there exists a formal power series of differential operators 
	$T= \operatorname{id} + \sum_{k=1}^\infty\lambda^k T_k$, with $T(1) =1$ such that $T(f) \star T(g) = T(f \star' g) $
	for all $f,g\in \CCinf(X)\ph$. The operator $T$ in the above definition is always invertible and indeed, given a star product $\star$,
	$f \star' g := T^{-1}(T(f) \star T(g))$ always gives a new equivalent star product. Two equivalent star products clearly give rise to the same Poisson bracket.\\
	For the star product (\ref{EqDefStandardStarProduct})
	there is the following well-known transformation
	$T=e^{-\lambda \Delta}$ with 
	$\Delta(f)=\partial^2f/\partial x\partial p$: together with the $\mathbb{K}$-linear (and not $\mathbb{K}[[\lambda]]$-linear) involution
	$L:A[[\lambda]]\to A[[\lambda]]$ given by
	$L\left(\sum_{r=0}^\infty\lambda^rf_r\right)=
	\sum_{r=0}^\infty(-\lambda)^rf_r$ we get --setting 
	$V=L\circ T$
	\begin{equation} \label{EqDefNeumaierModified}
	    \big(V(f)\big)\star_s 
	    \big(V(g)\big)
	    =V\big(g\star_s f\big)
	\end{equation}
	which can easily be checked on exponential functions
	$(x,p) \mapsto e^{ax+bp}$ with $a,b\in \mathbb{K}$.
	

	

	\section{Noncommutative localization of smooth star products on open subsets}
	\label{SecNonComLocOpenSets}
	
	Let $(X,\pi)$ be a Poisson manifold, let $\star=
	\sum_{k=0}^\infty \lambda^k C_k$ be a star product on
	$(X,\pi)$, and let $\Omega\subset X$ be a fixed open set.
	We set $K=\mathbb{K}[[\lambda]]$, and consider the $K$-algebra $\big(R=\mathcal{C}^\infty(X)[[\lambda]],\star\big)$.
	Moreover, since the star product $\star$ only involves
	bidifferential operators, it restricts to a star product $\star_\Omega$ on formal 
	power-series $\phi\in R_\Omega:=
	\mathcal{C}^\infty(\Omega,\mathbb{K})[[\lambda]]$ such that
	$\big(R_\Omega,\star_\Omega\big)$ is also a $K$-algebra.
	It follows that the restriction map 
	$\eta_\Omega=\eta:R\to R_\Omega:f\mapsto f|_\Omega$ is a morphism
	of unital $K$-algebras. We define the following subsets
	$S_\Omega\subset \mathcal{C}^\infty(X,\mathbb{K})$ and $S\subset R$:  
	\begin{equation}
		S_\Omega:=\left\{ g_0\in \mathcal{C}^\infty(X,\mathbb{K})~|~\forall~x\in\Omega:~g_0(x)\neq 0 \right\}
		~~\mathrm{and}~~S:=S_\Omega+\lambda R.
	\end{equation}
	Clearly, $S_\Omega$ is a commutative multiplicative subset of $\mathcal{C}^\infty(X,\mathbb{K})$. Since the constant function $1$ is in $S$, and for any
	$g,h\in S$ we have $(g\star h)_0(x)=g_0(x)h_0(x)\neq 0$ (for all
	$x\in \Omega$) it follows that \emph{$S$ is a multiplicative subset of
		the unital $K$-algebra $R$}.\\
	We can now consider the noncommutative localization of
	$R$ with respect to $S$ and compare it with the
	unital $K$-algebra $R_\Omega$:
	
	\begin{theorem}\label{TLocalizationEquivalenceOpenSubset}
		Using the previously fixed notations we get
		for any open set $\Omega\subset X$:
		\begin{enumerate}
			\item $(R_\Omega,\star_\Omega)$ equipped with the restriction morphism
			$\eta$ consitutes a right $K$-algebra of fractions for $(R,S)$.
			\item As an immediate consequence we have that $S$ is a right denominator set.
			\item \textbf{This implies in particular that the algebraic localization $RS^{-1}$
				of $R$ with respect to $S$ is isomorphic to 
				the concrete localization $R_\Omega$ as unital $K$-algebras.}
		\end{enumerate}	
	\end{theorem}
	\begin{proof} \textbf{1.} We have to check properties
		$(i.a.)$, $(i.b.)$, and $(i.c.)$ of Definition
		\ref{defi: bigdef}:\\
		$\bullet$~``\emph{$\eta$ is $S$-inverting}'' 
		(property $(i.a.)$): indeed,
		this is a classical reasoning from deformation quantization
		which we shall repeat for the convenience of the reader. Let
		$g\in S$ and $\gamma=\eta(g)$ its restriction to $\Omega$.
		Take $\psi\in R_\Omega$ and try to solve the equation
		$\gamma\star_\Omega \psi =1$. At order $k=0$ we get the
		condition $\gamma_0\psi_0=1$, but since $\gamma_0(x)\neq 0$
		for all $x\in \Omega$ the function 
		$x\mapsto \psi_0(x):=\gamma_0(x)^{-1}$ 
		is well-defined and smooth
		in $\mathcal{C}^\infty(\Omega,\mathbb{K})$. 
		Suppose by induction that the functions $\psi_0,\ldots,\psi_k\in
		\mathcal{C}^\infty(\Omega,\mathbb{K})$ have already been found in order to satisfy equation
		$\gamma\star_\Omega \psi =1$ up to order $k$. At order $k+1\geq 1$ the condition reads
		\[
		0 = \big(\gamma\star_\Omega \psi\big)_{k+1}
		=   \sum_{ \scriptsize \begin{array}{c}
				l,p,q=0 \\
				l+p+q=k+1
			\end{array}
		}^{k+1} C_l(\gamma_p,\psi_q)
		=  \gamma_0\psi_{k+1} 
		+F_{k+1}(\psi_0,\ldots,\psi_k,\gamma_0,\ldots,
		\gamma_{k+1})
		\]
		where the term starting with $F_{k+1}$ denotes the difference 
		$\big(\gamma\star_\Omega \psi\big)_{k+1}-\gamma_0\psi_{k+1}$
		which obviously does not contain $\psi_{k+1}$. Again,
		since $\gamma_0$ is nowhere zero on $\Omega$ the function 
		$\psi_{k+1}$ can be computed from this equation by
		multiplying with $x\mapsto \gamma_0(x)^{-1}$. Hence there is a solution $\psi\in R_\Omega$ of equation 
		$\gamma\star_\Omega \psi =1$. In a completely analogous way
		there is a solution $\psi'\in R_\Omega$ of the equation 
		$\psi' \star_\Omega\gamma =1$. By associativity of $\star_\Omega$ we get $\psi=\psi'$ as the unique inverse of $\gamma$ in the unital $K$-algebra $R_\Omega$.\\
		$\bullet$~``\emph{Every $\phi\in R_\Omega$ is equal to
			$\eta(f)\star_\Omega\eta(g)^{\star_\Omega -1}$ for some $f\in R$ and $g\in S$}'' (property $(i.b.)$): the main idea is to transfer the proof of Lemme 6.1 of Jean-Claude Tougerons's book \cite[p.113]{Tou72} to the
		non-commutative situation. Let 
		$\phi=\sum_{i=0}^\infty \lambda^i\phi_i\in R_\Omega$. We then fix the following data
		which we get thanks to the fact that $X$ and therefore
		each open set $\Omega$ is a second countable locally compact topological space: there is a sequence of compact sets
		$(K_n)_{n\in\mathbb{N}}$ of $X$, a sequence of open sets
		$(W_n)_{n\in\mathbb{N}}$,
		and a sequence of smooth functions $(g_n)_{n\in\mathbb{N}}:X\to \mathbb{R}$
		such that
		\[
		\bigcup_{n\in\mathbb{N}}K_n=\Omega,
		\]
		and
		\[
		\forall~n\in\mathbb{N}: K_n\subset W_n\subset \overline{W_n}\subset  K_{n+1}^\circ
		~~\mathrm{and}~~g_n(x)=\left\{\begin{array}{lc}
			1 & \mathrm{if}~x\in W_n, \\
			0 & \mathrm{if}~x\not\in K_{n+1}, \\
			y\in [0,1] & \mathrm{else}.
		\end{array}\right. .
		\]
		We denote by $\gamma_j$ the restriction $\eta(g_j)$ of
		$g_j$ to $\Omega$ for each nonnegative integer
		$j$.
		The idea is to define the denominator function $g$
		as a (non formal!) converging sum $g=\sum_{j=0}^\infty \epsilon_jg_j$.
		Choose a sequence $(\epsilon_j)_{j\in\mathbb{N}}$ of strictly positive real numbers such that
		\[
		\forall~j\in\mathbb{N}:~~ \epsilon_jp_{K_{j+1},j}(g_j)< \frac{1}{2^j}~~~\mathrm{and}~~~
		\forall~i\leq j\in\mathbb{N}:~~ \epsilon_j\sum_{l=0}^i
		p_{K_{j+1},j}\big(C_l(\phi_{i-l},g_j)\big)<
		\frac{1}{2^j}~
		\]
		(see eqn (\ref{EqDefSeminorms}) for the definition of the seminorms
		$p_{K,m}$)
		which is possible since for each nonnegative integer $j$ there are only finitely many seminorms and functions involved. For all nonnegative integers $i,j,N$ we define the functions
		$g_{(N)}\in \mathcal{C}^\infty(X,\mathbb{K})$, and
		$\psi_{ij},\psi_{(i,N)}\in \mathcal{C}^\infty(\Omega,\mathbb{K})$:
		\[
		g_{(N)}:=\sum_{j=0}^N\epsilon_jg_j,~~~
		\psi_{ij}:=\sum_{l=0}^iC_l\big(\phi_{i-l},\gamma_{j}\big),~~~
		\psi_{(i,N)}:=\sum_{j=0}^N\epsilon_j\psi_{ij}
		=\sum_{l=0}^iC_l\big(\phi_{i-l},\gamma_{(N)}\big),
		\]
		and since $\mathrm{supp}(g_{j})\subset K_{j+1}\subset \Omega$, hence
		$\mathrm{supp}(g_{(N)})\subset K_{N+1}\subset \Omega$,
		there are unique functions
		$f_{ij}\in\mathcal{C}^\infty(X,\mathbb{K})$ such that
		\[
		f_{ij}(x) := \left\{ \begin{array}{cl}
			\psi_{ij}(x) & \mathrm{if}~x\in \Omega,\\
			0            & \mathrm{if}~x\not\in\Omega.
		\end{array}\right.~,~~~
		\mathrm{hence}~~\eta(f_{ij})=\psi_{ij}~~~
		\mathrm{and}~~~\mathrm{supp}(f_{ij})\subset K_{j+1}.
		\]
		For each nonnegative integer $N$ we set
		$f_{(i,N)}:=\sum_{j=0}^N\epsilon_jf_{ij}\in 
		\mathcal{C}^\infty(X,\mathbb{K})$ with 
		$\mathrm{supp}(f_{(i,N)})\subset K_{N+1}$. Clearly,
		$\eta(f_{(i,N)})=\phi_{(i,N)}$.\\
		We shall now prove that both sequences $(g_{(N)})_{N\in\mathbb{N}}$, and for each  nonnegative integer $i$, $(f_{(i,N)})_{N\in\mathbb{N}}$ are Cauchy
		sequences in the complete metric space  $\mathcal{C}^\infty(X,\mathbb{K})$. First, it is obvious that for any two compact subsets $K,K'$ and nonnegative integers $N,N'$
		we always have for all $f\in\mathcal{C}^\infty(\mathbb{R}^n,\mathbb{K})$
		\begin{equation}
			\mathrm{if}~K\subset K'~\mathrm{and}~m\leq m'~\mathrm{then}~~ p_{K,m}(f)\leq p_{K',m'}(f).
		\end{equation}
		Fix a nonnegative integer $i$. 
		Let $\epsilon\in\mathbb{R}$, $\epsilon>0$, $K\subset X$  a compact subset, and $m\in\mathbb{N}$.
		Then there is a nonnegative  integer $N_0$ such that 
		\[
		\frac{1}{2^{N_0}}<\epsilon,~~~m\leq N_0,~~~\mathrm{and}~~i\leq N_0.
		\]
		Then for all nonnegative integers $N,p$ with $N\geq N_0$ we get (since for all $j\in\mathbb{N}$
		such that $N+1\leq j$ we have
		$m\leq N_0\leq N\leq j$ and $i\leq N$,
		and $\mathrm{supp}(f_{i,j})\subset K_{j+1}^\circ \subset K_{j+1}$)
		\begin{eqnarray*}
			p_{K,m}\big(f_{(i,N+p)}-f_{(i,N)}\big) & = &
			p_{K,m}\left(\sum_{j=N+1}^{N+p}\epsilon_jf_{i,j}\right) 
			\leq  \sum_{j=N+1}^{N+p}\epsilon_jp_{K,m}\big(f_{i,j}\big) 
			=\sum_{j=N+1}^{N+p}\epsilon_j
			p_{K\cap K_{j+1},m}\big(\psi_{ij}\big)\\
			& \leq & \sum_{j=N+1}^{N+p}\epsilon_j
			p_{K_{j+1},j}\left(\sum_{l=0}^iC_l(\phi_{i-l},g_j)\right)
			\leq 
			\sum_{j=N+1}^{N+p}\epsilon_j\sum_{l=0}^i
			p_{K_{j+1},j}\left(C_l(\phi_{i-l},g_j)\right)
			\\
			& < & \sum_{j=N+1}^{N+p} \frac{1}{2^j} =\frac{1}{2^N}\left(1-\frac{1}{2^p}\right) 
			<  \frac{1}{2^N} \leq \frac{1}{2^{N_0}} <\epsilon.
		\end{eqnarray*}
		It follows that for each $i\in\mathbb{N}$ the sequence $(f_{(i,N)})_{N\in \mathbb{N}}$ is a  Cauchy 
		sequence in the locally convex
		vector space $\mathcal{C}^\infty(X,\mathbb{K})$ hence converges to a smooth function
		$f_i=\sum_{j=0}^\infty \epsilon_j f_{i,j}$. Replacing in the above reasoning the function
		$\phi_0$ by the constant function $1$ on $\Omega$ it follows that
		the sequence $(g_{(N)})_{N\in\mathbb{N}}$ converges to a smooth function 
		$g:X\to \mathbb{R}$. Now let $x\in \Omega$. Then there is a nonnegative integer
		$j_0$ such that $x\in K_{j_0}$. It follows from the nonnegativity and the definition of all the $g_j$ and
		from the strict positivity of $\epsilon_j$ that
		\begin{equation}\label{EqCompFonctionApplatisseur}
		g(x)=\sum_{j=0}^\infty \epsilon_j g_j(x)\geq \epsilon_{j_0} g_{j_0}(x)=\epsilon_{j_0}>0
		\end{equation}
		showing that $g$ takes strictly positive values on $\Omega$
		whence $g\in S$.\\
		Now let $x\not \in\Omega$. Then for any
		$v\in T_xX$ with $h(v,v)\leq 1$
		we have that 
		\[
		\forall~m\in\mathbb{N}:~~ (D^mg_{(N)})(v)=\sum_{j=0}^N \epsilon_j (D^mg_{j})(v)=0
		\]
		because each $g_j$
		has compact support in $K_{j+1}\subset \Omega$. Since 
		$g_{(N)}\to g$ for
		$N\to \infty$ it follows by the continuity of differential operators and evaluation functionals that 
		$D^m g_{(N)}(v)\to D^m g(v)$,
		and hence 
		\begin{equation}\label{EqCompGFlatOutsideOmega}
			\forall~x\in X\setminus \Omega,~\forall~m\in \mathbb{N},~\forall~v\in T_xX,
			~h(v,v)\leq 1:~~(D^m g)(v)=0,
		\end{equation}
		and in a completely analogous manner
		\[
		\forall~x\in X\setminus \Omega,~\forall~m\in \mathbb{N},~\forall~v\in T_xX,
		~h(v,v)\leq 1:~~(D^m f_{i})(v)=0.
		\]
		Hence the infinite jets of all the functions $g$ and $f_i$, $i\in\mathbb{N}$,
		vanish outside the open subset $\Omega$. J.-C.~Tougeron
		calls the function $g$ \emph{fonction aplatisseur} for the
		family $(\phi_i)_{i\in\mathbb{N}}$ in case $C_l=0$ for
		$l\geq 1$.\\
		Now we get
		\[
		\left(\phi\star_U \eta(g_{(N)})\right)_i
		=\sum_{l=0}^iC_l\big(\phi_{i-l}, \eta(g_{(N)})\big)
		=\psi_{(i,N)}=\eta(f_{(i,N)}).
		\]
		Since the restriction map $\eta:\mathcal{C}^\infty(X,\mathbb{K})\to
		\mathcal{C}^\infty(\Omega,\mathbb{K})$ is continuous
		(where the Fr\'{e}chet topology on $\mathcal{C}^\infty(\Omega,\mathbb{K})$ is induced by those
		seminorms $p_{K,m}$ where $K\subset \Omega$) as are the bidifferential operators $C_l$ we can pass to the limit
		$N\to \infty$ in the above equation and get
		\[
		\phi\star_\Omega \eta(g) = 
		\sum_{i=0}^\infty \lambda^i 
		\big(\phi\star_\Omega \eta(g)\big)_i
		= \sum_{i=0}^\infty \lambda^i\eta(f_i)=:\eta(f).
		\]
		Since $g\in S$ it follows that $\eta(g)$ is invertible in
		$R_\Omega$ by property $(i.a)$ of Definition 
		\ref{defi: bigdef}, and the preceding equation implies
		$\phi=\eta(f)\star_\Omega \eta(g)^{\star_\Omega -1}$ thus proving property $(i.b)$ of Definition \ref{defi: bigdef}.\\
		$\bullet$~\emph{The kernel of $\eta$ is equal to the space
			of functions $f\in R$ such that there is $g\in S$
			with $f\star g=0$} (property $(i.c)$ of Definition 
		\ref{defi: bigdef}. Clearly if there is $f\in R$ and $g\in S$
		such that $f\star g=0$ then $\eta(f)\star_\Omega \eta(g)=0$,
		and since $\eta(g)$ is invertible in $R_\Omega$ we have
		$\eta(f)=0$.\\
		Conversely, if $f\in R$ such that $\eta(f)=0$, then for
		all integers $i\in\mathbb{N}$ and for all $x\in \Omega$
		we have $f_i(x)=0$. Hence the infinite jet of each $f_i$ vanishes
		at each point $x\in \Omega$ since $\Omega$ is open. Take
		the \emph{fonction aplatisseur} $g\in S$ constructed
		in the preceding part of the proof for $\phi_0=1,\phi_i=0$
		for all $i\geq 1$. Then we get
		\[
		\forall~x\in X:~(f\star g)_i(x)
		=\sum_{l=0}^i C_l(f_{i-l},g)(x)=
		\left\{\begin{array}{cl}
			0 & \mathrm{if~}x\in\Omega~\mathrm{since~every~jet  ~of~each}f_i~\mathrm{vanishes~in~}\Omega, \\
			0 & \mathrm{if~}x\not\in\Omega~
			\mathrm{since~every~jet~of~}g
			~\mathrm{vanishes~outside~of~}\Omega,
		\end{array}\right.
		\]
		where we have used eqn (\ref{EqCompGFlatOutsideOmega}) for the second alternative of the above statement.
		This proves part \textbf{1.} of the theorem.\\
		Statements \textbf{2.} and \textbf{3.} are immediate consequences of \textbf{1.} and Theorem \ref{TLocalizationForRightDenominatorSets}.
	\end{proof}
	
	\noindent \textbf{Remarks:} 
	For zero Poisson structure and
	trivial deformation $C_l=0$ for all $l\geq 1$ the above
	result specializes upon restricting to terms of order 0 to the classical result that algebraic
	and analytic localization with respect to an open subset
	$\Omega\subset X$ are isomorphic for 
	the commutative $\mathbb{K}$-algebra 
	$\mathcal{C}^\infty(X,\mathbb{K})$. \\
	Moreover, since for any
	closed set $F\subset X$ Tougeron's above construction gives us a smooth function $g:X\to \mathbb{R}$ which is nowhere zero on the open set $\Omega=X\setminus F$ and zero outside $\Omega$, hence on $F$, one gets the well-known result that
	the Zariski topology on $X$ induced by the commutative
	$\mathbb{K}$-algebra  $\mathcal{C}^\infty(X,\mathbb{K})$ coincides with 
	the usual manifold topology because each set $Z(I)$ is closed by continuity of all the functions in the ideal $I$, and conversely every closed set $F$ is the
	zero set $Z(gA)$ of the ideal $gA$ 
	(where $A=\mathcal{C}^\infty(X,\mathbb{K})$).\\
	Finally, note that the numerator morphism $\eta$ is injective iff the open set $\Omega$ is dense which is quite easy to see.

	\section{Noncommutative germs for smooth star products}
	\label{SecNoncommutative germs for smooth star products}
	
	Let $(X,\pi)$ again be a Poisson manifold, and let
	$\star=\sum_{l=0}^\infty\lambda^l
	C_l$ be a bidifferential 
	star product. Let $K=\mathbb{K}[[\lambda]]$, and we denote the unital $K$-algebra $\big(\mathcal{C}^\infty(X,\mathbb{K})[[\lambda]],\star\big)$
	by $R$.
	For any open set $U\subset X$ let $R_U$ denote the
	unital $K$-algebra $\big(\mathcal{C}^\infty(U,\mathbb{K})[[\lambda]],
	\star_U\big)$,
	where $\star_U$ denotes the obvious action of the bidifferential operators in $\star$ to the local functions
	in $\mathcal{C}^\infty(U,\mathbb{K})$. We write $R_X=R$. For any two open sets with
	$U\supset V$, denote by $\eta^U_V:R_U\to R_V$ be the restriction morphism where we write $\eta_U$ for $\eta^X_U$. Clearly, for $U\supset V\supset W$ one has the categorical identities $\eta^V_W\circ \eta^U_V=\eta^U_W$ and $\eta^U_U=\mathrm{id}_U$. Denoting by
	$\underline{X}$ the topology of $X$ it is readily checked that
	the family $\big(R_U\big)_{U\in \underline{X}}$ with the
	restriction morphisms $\eta^U_V$ defines a \emph{sheaf of $K$-algebras over $X$}, see e.g.~the book
	\cite{KS06} for definitions.\\
	Let $x_0$ a fixed point in $X$, and let 
	$\underline{X}_{x_0}\subset \underline{X}$ the set of all
	open sets containing $x_0$. We recall the definition of
	the \emph{stalk at $x_0$}, $R_{x_0}$ of the sheaf 
	$\big(R_U\big)_{U\in \underline{X}}$ whose elements are called \emph{germs at $x_0$}: it is defined as the inductive
	limit (or colimit, see \cite{Mac98}) $\lim_{U\in \underline{X}_{x_0}} R_U$. In order to perform computations we recall the more down-to-earth definition: let
	$\tilde{R}_{x_0}$ be the disjoint union of all the $R_U$, i.e.~ the set of all pairs $(U,f)$ where $U$ is an open
	set containing $x_0$ and $f\in \mathcal{C}^\infty(U,\mathbb{K})[[\lambda]]$. Define an 
	addition $+$ and a multiplication $\star$ on these pairs
	by
	\[
	(U,f)+(V,g)
	:= 
	\big(U\cap V,\eta^U_{U\cap V}(f)+\eta^V_{U\cap V}(g)\big)
	~~~\mathrm{and}~~~
	(U,f)\star(V,g)
	:= 
	\big(U\cap V,
	\eta^U_{U\cap V}(f)\star_{U\cap V}\eta^V_{U\cap V}(g)\big),
	\]
	and it is easily checked that the addition is associative and
	commutative, that the multiplication is associative, and that
	there is the distributive law. Furthermore, 
	the sum of $(U,f)$
	and $(V,0)$ equals $\big(U\cap V, \eta^U_{U\cap V}(f)\big)$
	which is equal to $(U,f)\star (V,1)=(V,1)\star (U,f)$. Next
	the binary relation $\sim_{x_0}$ defined by
	\[
	(U,f)\sim_{x_0} (V,g)~~\mathrm{iff}~~
	\exists~W\in\underline{X}_{x_0}~\mathrm{with~}
	W\subset U\cap V:~
	\eta^U_W(f)=\eta^V_W(g)
	\]
	turns out to be an equivalence relation. Denoting by $R_{x_0}$ the quotient set $\tilde{R}_{x_0}/\sim_{x_0}$
	and by $\eta^U_{x_0}:R_U\to R_{x_0}$ the restriction of the
	canonical projection $\tilde{R}_{x_0}\to R_{x_0}$ to
	$R_U\subset \tilde{R}_{x_0}$ (where $\eta^X_{x_0}$ will be shortened by
	$\eta_{x_0}:R\to R_{x_0}$) it is easy to see that the above
	addition and multiplication passes to the quotient, that
	all the zero elements $(U,0)$ are equivalent as are all the
	unit elements $(U,1)$, and that this defines the structure of a unital associative $K$-algebra
	denoted by $\big(R_{x_0},\star_{x_0}\big)$ on the quotient set such that
	all maps $\eta^U_{x_0}:\big(R_U,\star_U\big)\to \big(R_{x_0},\star_{x_0}\big)$ are morphisms of unital $K$-algebras. Note the following equations for all open sets
	$U\supset V$:
	\begin{equation} \label{EqCompFunctorialOnGerms}
		\eta^V_{x_0}\circ \eta^U_V = \eta^U_{x_0}.
	\end{equation}
	
	Define the following subsets $S=S(x_0)$ and $J=J_{x_0}$ of $R$:
	\begin{equation}
		S=S(x_0)=\left\{g\in R~|~g_0(x_0)\neq 0\right\}
		~~~\mathrm{and} ~~~
		J=J_{x_0}=\left\{g\in R~|~g_0(x_0)= 0\right\}.
	\end{equation}
	It is easy to see that $S=R\setminus J$, that $S$ is a \emph{multiplicative subset of $R$}, and that $J_{x_0}$ is a \emph{maximal ideal of $R$} (the quotient $R/J$ is isomorphic
	to the quotient $K/(\lambda K)\cong \mathbb{K}$ which is a field).\\
	We now have the following analog of Theorem
	\ref{TLocalizationEquivalenceOpenSubset}:
	
	\begin{theorem}\label{TLocalizationEquivalenceGerms}
		Using the previously fixed notations we get
		for any point $x_0\in X$:
		\begin{enumerate}
			\item $(R_{x_0},\star_{x_0})$ together with the morphism
			$\eta_{x_0}:R\to R_{x_0}$ consitutes a right $K$-algebra of fractions for $(R,S(x_0))$.
			\item As an immediate consequence we have that $S(x_0)$ is a right denominator set.
			\item \textbf{This implies in particular that the algebraic localization $RS^{-1}$
				of $R$ with respect to $S=S(x_0)$ is isomorphic to
				the concrete stalk $R_{x_0}$ as unital $K$-algebras.}
		\end{enumerate}	
	\end{theorem}
	\begin{proof}
		\textbf{1.} Once again, we have to check properties
		$(i.a.)$, $(i.b.)$, and $(i.c.)$ of Definition
		\ref{defi: bigdef}:\\
		$\bullet$~``\emph{$\eta_{x_0}$ is $S$-inverting}'' 
		(property $(i.a.)$): indeed, let $g\in S(x_0)$. Since
		$g_0(x_0)\neq 0$ there is an open neighbourhood $U$ of $x_0$ such that $g_0(y)\neq 0$ for all $y\in U$. Hence the restriction $\eta_U(g)$ is invertible in $(R_U,\star_U)$ by
		Theorem \ref{TLocalizationEquivalenceOpenSubset}. Using
		eqn (\ref{EqCompFunctorialOnGerms}) we see that $\eta_{x_0}(g)=\eta^U_{x_0}\big(\eta_U(g)\big)$, and the
		r.h.s.~is invertible in $R_{x_0}$ as the image of an invertible element $\eta_U(g)$ in $R_U$ with respect to the morphism of
		unital $K$-algebras $\eta^U_{x_0}$.\\
		$\bullet$~``\emph{Every $\phi\in R_{x_0}$ is equal to
			$\eta_{x_0}(f)\star_{x_0}\eta_{x_0}(g)^{\star_{x_0} -1}$ for some $f\in R$ and $g\in S(x_0)$}'' (property $(i.b.)$): indeed, let $\phi\in R_{x_0}$. By definition
		of $R_{x_0}$ as a quotient set there is an open neighbourhood $U$ of $x_0$ and an element $\psi\in R_U$
		with $\eta^U_{x_0}(U,\psi)=\phi$. According to the preceding
		Theorem \ref{TLocalizationEquivalenceOpenSubset} there are
		elements $f,g\in R$ with $g_0(y)\neq 0$ for all $y\in U$
		such that $\eta_U(f)=\psi\star_U\eta_U(g)$. In particular,
		$g_0(x_0)\neq 0$, hence $g\in S(x_0)$. Applying 
		$\eta^U_{x_0}$ to the preceding equation we get (upon using
		eqn (\ref{EqCompFunctorialOnGerms}))
		\[
		\eta_{x_0}(f)=\eta^U_{x_0}\big(\eta_U(f)\big)
		=\Big(\eta^U_{x_0}(\psi)\Big)\star_{x_0}
		\Big(\eta^U_{x_0}\big(\eta_U(g)\big)\Big)
		=\phi \star_{x_0} \big(\eta_{x_0}(g)\big)
		\]
		proving the result since $g\in S(x_0)$ and $\eta_{x_0}(g)$
		is invertible in the unital $K$-algebra $(R_{x_0},\star_{x_0})$.\\
		$\bullet$~\emph{The kernel of $\eta_{x_0}$ is equal to the space of functions $f\in R$ such that there is $g\in S(x_0)$
			with $f\star g=0$} (property $(i.c)$). Indeed, given $f\in R$ with $\eta_{x_0}(f)=0$ then there is an open neighbourhood $W$ of $x_0$ such that $\eta_W(f)=\eta_W(0)=0$. By the preceding Theorem
		\ref{TLocalizationEquivalenceOpenSubset} there is an element
		$g\in S_W\subset S(x_0)$ (which can be chosen to be a
		\emph{fonction aplatisseur}) such that $f\star g=0$. This proves \textbf{1.} of the theorem.\\
		\textbf{2.} and \textbf{3.} are immediate consequences of
		part \textbf{1.} and Theorem \ref{TLocalizationForRightDenominatorSets}.
	\end{proof}
	
	\noindent \textbf{Warning:} The stalk $R_{x_0}$ is taken in 
	the sense of sheaves of $\mathbb{K}[[\lambda]]$-algebras.
	Another interpretation would be two consider the
	sheaf 
	$\big(\mathcal{C}^\infty(U,\mathbb{K})
	\big)_{U\in\underline{X}_{x_0}}$ of commutative $\mathbb{K}$-algebras and the classical
	stalk $\mathcal{C}^\infty(X,\mathbb{K})_{x_0}$: in a completely analogous fashion it can be shown that it is 
	isomorphic to the algebraic localization with respect to
	the multiplicative set of functions which do not vanish
	at $x_0$. However the $\mathbb{K}[[\lambda]]$-module
	$\mathcal{C}^\infty(X,\mathbb{K})_{x_0}[[\lambda]]$ is NOT
	in general isomorphic to the above $R_{x_0}$: if 
	$f=\sum_{l=0}^\infty \lambda^lf_l$ is 
	a series of smooth functions such that $f_l$ vanishes
	on an closed ball of radius $\epsilon_l>0$ around $x_0$ where
	$\epsilon_l\to 0$ (for $l\to \infty$) and is non-zero outside, then the germ of each $f_l$ vanishes, but there
	is no common open neighbourhood of $x_0$ such that 
	$f$ restricted to that neighbourhood vanishes which would imply that the `$\mathbb{K}[[\lambda]]$-germ of
	$f$' vanishes. We shall come back to this problem in Section
	\ref{SecCommutativelyLocalizedStarProducts}.


	\section{Commutatively localized star products}
	 \label{SecCommutativelyLocalizedStarProducts}

	In this section we shall describe a more algebraic framework to generalize the two preceding sections.
	Let in the following $K$ be a fixed unital associative
	commutative ring. Unadorned tensor products $\otimes$
	are always with respect to $K$, hence $\otimes=\otimes_K$.
\subsection{Algebraic (multi)differential operators and
	    their localization}
	\label{SubSecAlgebraicMultidifferentialOperators}
	   	
	We shall first recall the well-known theory of 
	\emph{algebraic
	(multi)differential operators and their localization},
see e.g.~\cite{KLV86}, \cite{Nes03}, \cite{LR97},
     \cite[p.566-578]{Wal07}, and \cite{Vez97}: let $A$ be commutative associative unital $K$-algebra. We shall need the theory only for $A$ and its tensor products over $K$,
     but --as usual-- indulging in some more generality has the benefit of being more economic for the computations: let $M$ and $N$ be
     left $A$-modules. For each $a\in A$ fix the following
     $K$-linear maps $L_a$, $R_a$, and $\mathrm{ad}_a$ 
     from the $K$-module $\mathrm{Hom}_K(M,N)$ to itself defined in the following way for all $\phi\in 
     \mathrm{Hom}_K(M,N)$ and $m\in M$:
     \begin{equation}
      \big(L_a(\phi)\big)(m)=a\big(\phi(m)\big),~~
      \big(R_a(\phi)\big)(m)=\phi(am),~~
      \mathrm{ad}_a(\phi)=L_a(\phi)-R_a(\phi)
     \end{equation}
     which obviously all commute.
     Then a $K$-linear map $\phi:M\to N$ is called a
     \emph{differential operator of order $k\in\mathbb{N}$ with respect to the $K$-algebra $A$} iff for all $a_1,\ldots,a_{k+1}\in A$ we have
     $\left(\mathrm{ad}_{a_1}\circ\cdots\circ \mathrm{ad}_{a_{k+1}}\right)(\phi)=0$. It is well-known
     that the set of all differential operators of order $k$
     forms an $A$-$A$-bimodule (w.r.t.~$L_a$ and $R_a$), and
     that these bimodules form an increasing filtration (indexed by the order) of
     the $A$-$A$-bimodule $\mathrm{Hom}_K(M,N)$ whose union
     in $\mathrm{Hom}_K(M,N)$ is called the $A$-$A$-bimodule of all differential operators. The
     $A$-$A$-bimodule of all differential operators of order
     $0$ is clearly identical to the set of all 
     $A$-linear maps. Moreover, the composition $\psi\circ \phi$
     of a differential operator $\phi:M\to N$ of order 
     $k_1$ and a differential operator $\psi:M\to P$ (where
     $P$ is another $A$-module) of order $k_2$ is a differential operator $M\to P$ of order $k_1+k_2$.
     Therefore there is a category $A\mathbf{-Moddiff}$
     whose objects are $A$-modules and morphisms differential
     operators.
     Let $A'$ be another unital associative commutative
     $K$-algebra, and $M'$, $N'$ be $A'$-modules. If
     $\phi:M\to N$ and $\phi':M'\to N'$ are differential operators of order $k$ and $k'$, respectively, with respect to $A$ and $A'$, respectively, then
     \begin{equation}\label{EqCompTensorProductOfDiffOps}
     	\phi\otimes \phi':
     	      M\otimes M'\to N\otimes N'
     	      ~~\mathrm{is~a~differential~operator~
     	      of~order~}k+k'~
     	      \mathrm{with~respect~to~}A\otimes A',
     \end{equation}
     which follows from the obvious equation
     $\mathrm{ad}_{a\otimes a'}(\phi\otimes \phi')
     =\left(\mathrm{ad_a}(\phi)\right)
       \otimes \left(R_{a'}(\phi')\right)
       +\left(L_a(\phi)\right)
       \otimes \left(\mathrm{ad}_{a'}(\phi')\right)$
       for all $a,a'\in A$, and its iterations.
     Moreover, if $\chi:A\to A'$ is a $K$-algebra morphism
     and $\phi':M'\to N'$ a differential operator of order
     $k'$ with respect to $A'$ it is obvious that $\phi'$ is also
     a differential operator of the same order $k'$ with respect to $A$ whence there is an obvious \emph{restriction functor}
     from $A'\mathbf{-Moddiff}$ to $A\mathbf{-Moddiff}$.
     In the particular case of $A'=A_S$, the algebra of quotients of $A$ with respect to a fixed multiplicative subset $S\subset A$, and $\chi=\eta$, the numerator morphism, this restriction functor has a left adjoint which amounts to the \emph{localization of differential operators} as has been shown by G.Vezzosi in his PhD-thesis, see \cite[Prop.~3.3]{Vez97}:
     \begin{theorem}[G.Vezzosi, 1997] 
     	\label{TVezzosi}
     	Given the $K$-algebra $A$ and
     	the multiplicative subset $S$ there is a covariant functor
     	$(~)_S:A\mathbf{-Moddiff}\to A_S\mathbf{-Moddiff}$
     	which is left adjoint to the above restriction
     	functor $A\mathbf{-Moddiff}\leftarrow A_S\mathbf{-Moddiff}$ induced by the numerator morphism $A\to A_S$: on objects it is given by
     	the localization of modules $M\to M_S$, and
     	each differential operator $D:M\to N$ of order $k$
     	w.r.t.~$A$ is mapped to the following differential
     	operator $D_S:M_S\to N_S$ of the same order $k$
     	w.r.t.~$A_S$
     	defined as follows for all $m\in M$ and $s\in S$
     	\begin{equation}\label{EqDefDifferentialOpLocalized}
     	  D_S\left(\frac{m}{s}\right)
     	  = \sum_{r=1}^{k+1}{k+1 \choose r}(-1)^{r+1}
     	      \frac{D\big(s^{r-1}m\big)}{s^r}.
     	\end{equation}
     	In particular, $D_S$ is uniquely determined by
     	its values $D_S\left(\frac{m}{1}\right)=
     	\frac{D(m)}{1}$ for all $m\in M$, and it follows that $(D\circ D')_S
     	=D_S\circ D'_S$ whenever the composition $D\circ D'$
     	makes sense.
     \end{theorem}
    \noindent The proof is quite technical: eqn (\ref{EqDefDifferentialOpLocalized}) is motivated by
    the fact that if $D_S:M_S\to N_S$ is a differential operator of order $k$ satisfying $D_S(m/1)=D(m)/1$ then --by definition-- it satisfies $0=-(1/s)^{k+1}\big(\ad_{s/1}^{k+1}(D_S)\big)(m/s)$
    for all $m\in M$ and $s\in S$ which gives
     eqn (\ref{EqDefDifferentialOpLocalized}).
     The right hand side of eqn (\ref{EqDefDifferentialOpLocalized}) can be defined for any $K$-linear map $M\to N$ as a
    set-theoretic map $M\times S\to N_S$, and the fact that
    it only depends (first in a set-theoretical way) on the fraction $\frac{m}{s}$ is shown
    by induction over the order of the differential  operator $D$. Note also that
    it can be shown a posteriori that the integer $k$ 
    in eqn (\ref{EqDefDifferentialOpLocalized}) can
    be replaced by any integer $k'\geq k$ without changing
    the left hand side.
    
    Next, let $p$ be a positive integer, let $M_1,\ldots,M_p$,
    $N$ be $A$-modules, and $\mathsf{k} =(k_1,\dots ,k_p) \in 
    \mathbb{N}^p$ a multi-index. Recall that a $K$-linear map
    $C:M:=M_1\otimes\cdots\otimes M_p\to N$ is called a \emph{multidifferential operator of rank} $p$ \emph{of order}
    $\mathsf{k}$ \emph{with respect to $A$} --which is sometimes also called a \emph{polydifferential operator}-- iff for each integer $1\leq i\leq p$ and
    for all $m_1\in M_1,\ldots,m_{i-1}\in M_{i-1}$, $m_{i+1}\in M_{i+1},\ldots,m_p\in M_p$ the $K$-linear map
    $M_i\to N$ given by $m_i\mapsto C(m_1\otimes\cdots\otimes m_p)$ is a differential operator of order $k_i$. For the particular case $A = \CCinf(X)$ for a smooth manifold $X$ this algebraic definition is well-known to coincide with the analytic definition, see e.g.~\cite[p.~575, Satz A.5.2.]{Wal07} which means that in local charts an (algebraically defined) multidifferential operator looks as in eqn (\ref{EqDefMultiDifferentialOperatorsAn}).\\
    For our purposes it is more convenient to use the following formulation: note that the $K$-module
    $M_1\otimes\cdots\otimes M_p$ is a module with respect to
    the unital commutative associative $K$-algebra
    $A^{\otimes p}=A\otimes\cdots\otimes A$ ($p$ tensor factors) in a natural way, and that $N$ also can be viewed as a $A^{\otimes p}$-module by means of
    $(a_1\otimes\cdots\otimes a_p)n=a_1\cdots a_pn$
    for all $a_1,\ldots,a_p\in A$ and $n\in N$. Let 
    $\Phi:M_1\otimes\cdots\otimes M_p\to N$ be a
    $K$-linear map. If it is a differential operator of
    order $k$ with respect to $A^{\otimes p}$ it is easy to
    see by restricting to $1\otimes \cdots 1\otimes a_i\otimes 1\otimes \cdots \otimes 1 \in A^{\otimes p}$,
    $1\leq i\leq p$, $a_i\in A$, that $\Phi$ is a multidifferential operator of rank $p$ and order $(k,\ldots,k)$ with respect to $A$. Conversely, 
    for any $a\in A$ and any integer $1\leq r\leq p$
    writing $L_a, R^{r}_a, \mathrm{ad}^{r}_a$ for the
    following $K$-linear maps from
    $\mathrm{Hom}_K(M_1\otimes\cdots\otimes M,N)$
    to itself given by (for all $m\in M$) $(L_a(C))(m)
    =aC(m)$, $(R^{(r)}_a(C))(m)=C(a^{(r)}m)$
    (where 
    $a^{(r)}
    =1^{\otimes(r-1)}\otimes a \otimes
    	1^{\otimes (p-r)}$), and $\mathrm{ad}^{(r)}_a=L_a-
    	R^{(r)}_a$, there is the easy identity for all
    $a_1,\ldots,a_p\in A$
    \[
      \mathrm{ad}_{a_1\otimes\cdots\otimes a_p}
         =\sum_{r=1}^pL_{a_1}\circ \cdots\circ 
         L_{a_{r-1}}\circ 
         \mathrm{ad}^{(r)}_{a_{r}}\circ R^{(r+1)}_{a_{r+1}}
         \circ\cdots\circ R^{(p)}_{a_{p}}.
    \]
    By iteration this shows 
    that if $C$ is a multidifferential operator of rank $p$ and order $\mathsf{k}=(k_1,\ldots,k_p)$ w.r.t.~$A$ then $C$ is
    a differential operator of order $k_1+\cdots+k_p$ w.r.t.
    $A^{\otimes p}$. Hence
    \begin{equation}
    \big\{\mathrm{multi differential~operators~of~rank~}p
    \mathrm{~w.r.t.}~A\big\}
    =\big\{\mathrm{differential~operators}
    \mathrm{~w.r.t.}~A^{\otimes p}\big\}.
    \end{equation}
    With this identification, given a multiplicative subset
    $S\subset A$ it is now straight-forward to 
    localize multidifferential operators by localizing
    them as differential operators w.r.t.~$A^{\otimes p}$
    taking the multiplicative subset 
    $S^{\otimes p}\subset A^{\otimes p}$ (which is the obvious iteration of Remark 2 before eqn (\ref{EqCompLocalizationOfTensorProducts})) upon using Vezzosi's Theorem \ref{TVezzosi}. Note that it is easy
    to see that the localization of the $A$-module $N$
    w.r.t.~the multiplicative subset $S$ is naturally
    isomorphic to the localization of $N$ seen as a
    $A^{\otimes p}$-module w.r.t.~the multiplicative
    subset $S^{\otimes p}$.\\
    We are interested in the particular case where all the
    $A$ modules $M_1,\ldots,M_p,N$ are equal to $A$
    for which we state the preceding considerations in
    the following
    \begin{prop}\label{PLocMultiDiffOpsAAAAToA}
		Let $S_0 \subset A$ be a multiplicative subset, let
		$A_{S_0}$ be the ordinary commutative localization of
		$A$ w.r.t.~$S_0$, and let
		$\eta_{(A,S_0)} =\eta:A\to A_{S_0}$ be the numerator morphism. Let $C$ be a multidifferential operator of rank $p$ from $A^{\otimes p}$ to $A$.\\
		Then there exists a unique  multidifferential operator of rank $p$ , $C_{S_0}$, from $(A_{S_0})^{\otimes p}$ to $A_{S_0}$  such that $\eta\circ C = C_{S_0} \circ \eta^{\otimes p}$.\\
		Furthermore, given another multidifferential operator $C'$ of rank $p'$ we have 
		$(C \circ_i C')_{S_0} = C_{S_0} \circ_i C'_{S_0}$
		for each integer $1\leq i\leq p$.
	\end{prop}
	\begin{proof}
		The first part follows from the above considerations.
		The second part follows from the equation
		$C\circ_i C'=C\circ \big(
		\mathrm{id}^{\otimes (i-1)}\otimes C'\otimes 
		\mathrm{id}^{\otimes (p-i)}\big)$ seen as composition
		of differential operators w.r.t.~the $K$-algebra
		$A^{\otimes (p+p'-1)}$ and multiplicative subset
		$S_0^{\otimes (p+p'-1)}$ using eqn 
		(\ref{EqCompTensorProductOfDiffOps}).
	\end{proof}

	\subsection{Commutatively localized algebraic star products}
	\label{SubSecCommutativelyLocalizedStarProducts}
	
	Observe now that the Definition \ref{DefStarProducts} of star products can be generalized to any commutative associative unital $K$-algebra
	$A$ whence the significance `bidifferential' for the 
	$K$-bilinear maps $C_k:A\times A\to A$ is now given by the algebraic definition outlined in the preceding Section \ref{SubSecAlgebraicMultidifferentialOperators}.
	We have
	\begin{prop}
		Let $A$ be a commutative unital $K$-algebra 
		and a differential star product $\star = \sum_{i=0}^\infty \lambda^i  C_i$ on 
		$R:=A[[\lambda]]$ where the $C_i$ are bidifferential operators on $A$. For any multiplicative subset
		$S_0 \subset A$ there exists a 
		unique star product $\star_{S_0}$ on $A_{S_0}\ph$
		such that the numerator map $\eta$ canonically extended as a $K[[\lambda]]$-linear map  (also denoted
		$\eta$)
		$A\ph\to A_{S_0}\ph$  is a morphism of unital $K[[\lambda]]$-algebras. 
	\end{prop}
	\begin{proof}
		This follows from the previous Proposition \ref {PLocMultiDiffOpsAAAAToA} by considering the localization of the bidifferential operators $C_i$. It remains associative since the localization is compatible with the compositions $\circ_1$ and $\circ_2$.
	\end{proof}
	
	\noindent With the above structures $A,S_0,\star$ 
	we set $R=A[[\lambda]]$ and consider the following rather
	natural subset
	\begin{equation}
		S: = S_0 + \lambda A[[\lambda]] \subset R=A[[\lambda]].
	\end{equation}
	which can be called the \emph{canonical deformation of the multiplicative subset $S_0$}.
	Then we have the
	\begin{prop}\label{PSZeroSigmaCMImpliesAllsWell}
		The subset $S=S_0+\lambda R$ is a multiplicative subset
		of the algebra $(R,\star)$, and its image under 
		$\eta$ consists of invertible elements of the
		$K[[\lambda]]$-algebra
		$\big(A_{S_0}[[\lambda]], \star_{S_0}\big)$.\\
		It follows that there is a canonical morphism
		of unital algebras over $K[[\lambda]]$
		\begin{equation}
			\label{EqDefMorphismLocCommutesWithDeform}
			\Phi:\Big( \big(A[[\lambda]]\big)_S, \star_S\Big)\to
			\big( A_{S_0}[[\lambda]] ,\star_{S_0}\big).
		\end{equation}
		where the localization $\big(A[[\lambda]]\big)_S$
		is the general construction, see Proposition
		\ref{PGeneralLocalization}.
	\end{prop}
	\noindent Indeed, since the deformation terms of $\star$ come in
	higher orders of $\lambda$ it is clear that $S$ is multiplicative. Since $\eta(S_0)$ is invertible in
	$A_{S_0}$ this also holds for the image under $\eta$ of the canonical deformation $S$ of $S_0$, see the reasoning in the beginning of the proof of
	Theorem \ref{TLocalizationEquivalenceOpenSubset}
	which is completely algebraic.
	The existence of the algebra morphims $\Phi$ is then clear from the universal property of the localized algebra, see Proposition \ref{PGeneralLocalization}.
	
	Here we come to two general problems:
	\begin{center}
\textbf{1. Does localization commute with deformation ?}\\
     Meaning: is the above morphism $\Phi$ (\ref{EqDefMorphismLocCommutesWithDeform}) an 
     isomorphism?\\[2mm]
\textbf{2. Is $S$ a right (or left) denominator set?}
	\end{center}

	Note that even in the commutative case, i.e.~the localization of an algebra $R[[\lambda]]$ where $R$ is commutative, the map $\Phi$ is not always an isomorphism. This has already been noted in \cite{Arn73}.
	
For localization with respect to open sets (see section \ref{SecNonComLocOpenSets} ($S_0=S_\Omega$) the morphism
	$\Phi$ is an isomorphism, and $S=S_\Omega+\lambda R$ is a left and right denominator set. However, $\Phi$ is not injective
	for the germs ($S_0=A\setminus I_{x_0}$) in section
	\ref{SecNoncommutative germs for smooth star products})
	as the warning at the end of the section indicates
	although $S=S_0+\lambda R$ is a left and right denominator set.

	One reason why $\Phi$ is in general not an isomorphism is that $\big(A[[\lambda]]\big)_S$ is 
	in general no longer a \emph{topologically free  $K[[\lambda]]$-module}, see e.g.~\cite[p.388-391]{Kas95}
	for all the details.
	Given a $K[[\lambda]]$-module $M$ there is a natural topology with basis induced by the (descending) filtration $\{\lambda^k M \}_{k \in \N}$. 
The space $M$ is complete if for every sequence $(m_i ) \subset M$ the series $\sum_{i=0}^\infty m_i \lambda^i$ convergences in $M$. It is Hausdorff iff $\bigcap_{i=0}^\infty \lambda^i M = \{0\}$ iff $\{0\}$ is closed in the $\lambda$-adic topology. 
	A $K[[\lambda]]$-linear map between two $K[[\lambda]]$-modules is always continuous.
	 \\
Next, a $K[[\lambda]]$-module is called (topologically) free if it is isomorphic to a $K[[\lambda]]$-module of the form $V[[\lambda]]$ for some $K$-module $V$.  
	We have $V[[\lambda]] =  V \hat\otimes_K K[[\lambda]]$. Note that here  the tensor product is not the algebraic tensor product, but its completion in the $\lambda$-adic topology, see e.g.~\cite[p.~390-391]{Kas95}. Moreover recall that the $\lambda$-torsion  
	of a $K[[\lambda]]$ module $M$ is the set of elements  $m \in M$ for which $\lambda  m =0$.
	There is the following well-known characterization, see 
	e.g.~\cite[p.390, Prop.~XVI.2.4.]{Kas95}:
	\begin{prop}
		A $K[[\lambda]]$-module $M$ is topologically free if and only if it is complete and Hausdorff in the $\lambda$-adic topology and $\lambda$-torsion free. In this case $M \cong (M/\lambda M)[[\lambda]]$.
	\end{prop}
\noindent Since completeness and Hausdorffness are preserved by isomorphism, $\big(A[[\lambda]]\big)_S$ needs to be complete and Hausdorff for $\Phi$ to be an isomorphism. In fact, it needs to be topologically free. \\
If  $\big(A[[\lambda]]\big)_S$ is not Hausdorff, $\Phi$ is not injective, since then $\Phi^{-1}(0) \neq \{0\}$, since it  is a closed subset in the $\lambda$-adic topology.\\ 
The example of germs (Section \ref{SecNoncommutative germs for smooth star products}) is an example of this:\\
Consider the example at the end of Section \ref{SecNoncommutative germs for smooth star products}. Then for any $k  \in \N$, we have $\sum_{l=0}^k \lambda^l f_l = 0 \in R_{x_0}$ since it vanishes on the ball of radius $\epsilon_k$  around $x_0$ (if we choose the sequence $(\epsilon_l)$ monotone). This means $f = \lambda^k \sum_{l_0}^\infty \lambda^l f_{l-k}$ so $f \in \lambda^k R_{x_0}$ for all $k$ but as stated before $f \neq 0$.

\begin{prop}
	Consider the situation of Proposition \ref{PSZeroSigmaCMImpliesAllsWell}.
	If $(A[[\lambda]])_S$ is complete then the map $\Phi$ is surjective. 
\end{prop}
\begin{proof}
	Let $a_0 \in A$ and $s_0 \in S_0$.
	We have $\Phi(a_0 \star_S (s_0)^{\star-1}) = \Phi(a) \star_{S_0} \Phi(s^{\star-1}) =  \frac{a_0}{s_0} + \lambda r $ with $r = \frac{a_1}{s_1} \in (A)_{S_0}[[\lambda]]$.  Recursively one can find $a_i \in A, s_i \in S_0$, such that $\Phi(a_0 \star_S (s_0)^{\star-1} - \sum_{i=1}^\infty \lambda^i a_i \star_S (s_i)^{\star-1} =  \frac{a_0}{s_0} $.  The series on the left hand side converges since  we assume $A[[\lambda]]_S$ to be complete.
\end{proof}
More generally, it is always possible to extend the map $\Phi$ to the completion of $A[[\lambda]]_S$ due to the completeness of $(A_0)_{S_0}[[\lambda]]$ and continuity of $\Phi$. The previous proposition implies that this extension is surjective.

\noindent It may be interesting to develop a noncommutative
localization along the lines of Section \ref{SubSecAlgebraicLocalization}, in particular in the spririt of
Proposition \ref{PGeneralLocalization} and/or
Theorem \ref{TLocalizationForRightDenominatorSets}, for complete unital associative $K[[\lambda]]$-algebras
whose multiplicative subsets have some additional properties.

\subsection{A particular result generalizing the 
	restriction to open sets, Section \ref{SecNonComLocOpenSets}}

Let $A$ be a $K$-algebra. Suppose that the multiplicative
set $S_0\subset A$ has the following property
\begin{equation}
  \label{EqDefSigmaCM}
   \forall~\mathrm{~sequence~}(s_n)_{n\in \mathbb{N}}\in S_0~
   \exists~\mathrm{~sequence~}
    (b_n)_{n\in \mathbb{N}}\in A~\mathrm{and}~s\in S_0:
    \mathrm{s.~t.~}\forall~n\in\mathbb{N}:
    ~~s_nb_n=s.
\end{equation}
Note that for a sequence having only a finite number of pairwise different terms this is always trivially satisfied by choosing
for $s$ the common multiple of all the members in the associated finite set of the sequence. Moreover in the 
uninteresting case where $S_0$ contains $0$ the above
property (\ref{EqDefSigmaCM}) is trivially satisfied by
choosing the constant $0$-sequence for $(b_n)_{n\in\mathbb{N}}$.
Returning to the general case, we shall refer to property 
(\ref{EqDefSigmaCM}) as $\sigma CM$
(something like `countable common multiple').
A similar property has been considered in \cite{Gil67,She71}. However, there the common multiple $s$ is only considered to be different from 0.
Note however they consider domains, so $s$ is no zero divisor. This implies that one can consider the multiplicative set $S'$ generated by $S$ and $s$. Further the localization with respect to $S$ embeds injectively into the localization with respect to $S'$.
\begin{prop}
	For any open set $\Omega\subset X$ of a smooth manifold
	the multiplicative subset $S_\Omega=\{g\in A~|~\forall~x\in\Omega:~g(x)\neq 0\}$ appearing in 
	Section \ref{SecNonComLocOpenSets} has the 
	$\sigma CM$-property.
\end{prop}	
\noindent Indeed this follows from the proof of Theorem \ref{TLocalizationEquivalenceOpenSubset} in the trivial case where all the bidifferential operators of strictly positive order vanish, and where we set --for any given sequence $(s_n)_{n\in\mathbb{N}}$-- $\phi(x)=\sum_{n=0}^\infty
\lambda^n (1/s_n(x))$ for all $x\in \Omega$, and the
function applatisseur $g$ (see eqn (\ref{EqCompFonctionApplatisseur})) 
will be the desired element $s\in S_\Omega$. This construction is due to J.-C.Tougeron \cite{Tou65}.

The main result of this subsection is the following
\begin{prop}
  Suppose that the multiplicative subset $S_0\subset A$	
  satisfies the $\sigma CM$ property. \\
  Then the morphism
  $\Phi$, see eqn (\ref{EqDefMorphismLocCommutesWithDeform}), is an isomorphism, and the deformed multiplicative subset
  $S=S_0+\lambda A[[\lambda]]$ is right and left denominator subset of the algebra $R$.
\end{prop}
\begin{proof}
  We first note the following easy, but important property
  of general differential operators $D:M\to N$ of order $k$ where 
  $M$ and $N$ are arbitrary $A$-modules: for any $a\in A$ and
  $n\in\mathbb{N}$ with $n\geq k$ there are differential
  operators $\tilde{D}_{[a]},\check{D}_{[a]}:M\to N$
  of order $k$ such that for all $m\in M$
  \begin{equation}\label{EqCompDiffopsWithPowersMult}
    D(a^nm)=a^{n-k}\tilde{D}_{[a]}(m)
    ~~~\mathrm{and}~~~
    a^nD(m)=\check{D}_{[a]}(a^{n-k}m).
  \end{equation}
  Indeed write $R^n_a=\big(L_a-\mathrm{ad}_a)^n$ for the 
  term on the left of the first equation, and
  $L^n_a=\big(\mathrm{ad}_a+R_a)^n$ for the 
  term on the left of the second equation, apply the binomial theorem, and use
  that all maps $\mathrm{ad}_a^{l}(D)$ are differential
  operators of order $k-l\leq k$ and 
  $\mathrm{ad}_a^{k+1}(D)=0$.\\
  Next, let $\star=\sum_{n=0}^\infty\lambda^nC_n$ be the star product with algebraic bidifferential operators $C_n$,
  $n\in\mathbb{N}$. We can assume that each $C_n$ has a 
  `bi-order' $(k_n,k_n)$ with $k_n\in\mathbb{N}$ for each
  $n\in\mathbb{N}$ (of course $k_0=0$), and for each $n\in\mathbb{N}$ we define the nonnegative integer $\kappa_n:=\max\{k_0=0,k_1,\ldots,k_n\}$.\\
  We shall show that 
  $\big(A_{S_0}[[\lambda]],\star_{S_0}\big)$ is a right algebra of fractions of $\big(A[[\lambda]],\star,S\big)$
  along the (algebraisized) lines of the proof of Thm
  \ref{TLocalizationEquivalenceOpenSubset}:\\
  $\bullet$ It follows from the previous section
  (and from the beginning of the proof of Theorem \ref{TLocalizationEquivalenceOpenSubset})
  that the numerator morphism
  $\eta:A[[\lambda]]\to A_{S_0}[[\lambda]]$ is $S$-inverting.
  \\
  $\bullet$~``\emph{Every $\phi=\sum_{n=0}^\infty
  	 \lambda^n \frac{a_n}{s_n}\in A_{S_0}[[\lambda]]$
  	is equal to
  	$\eta(f)\star_{S_0}\eta(g)^{\star_{S_0} -1}$ for some $f=\sum_{n=0}^\infty\lambda^n\alpha_n\in A[[\lambda]]$ and
  $g\in S$}'': here of course $a_0,a_1,\ldots\in A$, 
$\alpha_0,\alpha_1,\ldots,\in A$, and
$s_0,s_1,\ldots \in S$. We make the ansatz $g=s\in S_0$ of a `fonction applatisseur', and consider
\begin{equation}\label{EqCompPhiStarSZeroSOverOne}
   \left(\phi \star_{S_0} \frac{s}{1}\right)_n
   = \sum_{u=0}^nC_{uS_0}\left(\frac{a_{n-u}}{s_{n-u}},
            \frac{s}{1}\right)
    \stackrel{(\ref{EqDefDifferentialOpLocalized})}{=}
    \sum_{u=0}^n\sum_{v=1}^{\kappa_n + 1}
        {\kappa_n+1 \choose v}(-1)^{v+1}
        \frac{C_u\big(s_{n-u}^{v-1}a_{n-u},s\big)}
                 {s_{n-u}^{v}}
\end{equation}
We have to choose $s\in S_0$ in such a way as to `kill the
denominators occurring on the right hand side of the preceding equation': thanks to the $\sigma CM$ property,
for the sequence 
$\left(
(s_0s_1\cdots s_n)^{2\kappa_n+1}\right)_{n\in\mathbb{N}}$
which is in $S_0$ there is a sequence 
$(b_n)_{n\in\mathbb{N}}$ and $s\in S_0$ such that for each $n\in\mathbb{N}$
we have 
\begin{equation}\label{EqDefKillingSequenceIdentity}
      \forall~n\in\mathbb{N}:~~
      (s_0s_1\cdots s_n)^{2\kappa_n+1}b_n=s.
\end{equation}
Clearly, in each of the numerators of the fractions on the right hand side of eqn
(\ref{EqCompPhiStarSZeroSOverOne}) the above $s$ can be
written as a product of $s_{n-u}^{2\kappa_n+1}c_{n,u}$ with
$c_{n,u}$ is a product of $b_n$ and some factors of the
above sequence. By the first equation of
(\ref{EqCompDiffopsWithPowersMult}) we can pull 
$s_{n-u}^{\kappa_n+1}$ out of the second argument of the bidifferential operator in
the numerator, and this factor in the numerator cancels each denominator.
This shows that there is $f\in A[[\lambda]]$ such that $\phi\star_{S_0}\eta(s)=\eta(f)$ and since
$\eta(s)$ is $\star_{S_0}$-invertible, the statement is
proved.\\
$\bullet$~\emph{The kernel of $\eta$ is equal to the space
	of elements $f\in R$ such that there is $g\in S$
	with $f\star g=0$:} indeed, the statement
   $f=\sum_{n=0}\lambda^n f_n\in A[[\lambda]]$ 
   is such that $\frac{f}{1}=\eta(f)=0$ is equivalent to the statement
   for each $n\in\mathbb{N}$ there is $s_n\in S_0$ such that
   $f_ns_n=0$. In order to get an idea of $g\in S_0$ we again
   make the ansatz $g=s\in S_0$ and we compute for each $n\in\mathbb{N}$
\begin{equation}\label{EqCompFTimesKillerFunction}
  (f\star s)_n=\sum_{u=0}^nC_u(f_{n-u},s).
\end{equation}
 We now take the same element $s$ constructed in the preceding part of the proof satisfying
 eqn (\ref{EqDefKillingSequenceIdentity}) with respect
 to the above $s_0,s_1,\ldots \in S_0$ each killing 
 $f_0,f_1,\ldots$. As in the preceding part, we can pull
 a factor $s_{n-u}^{\kappa_n+1}$ out of the second argument
 of the bidifferential operator $C_u$ (upon using the first equation of eqn (\ref{EqCompDiffopsWithPowersMult})), and we
 put it then into the first argument of the resulting
 bidifferential operator where a factor of $s_{n-u}$ remains
 in front of $f_{n-u}$ which gives zero (upon using the
 second equation of (\ref{EqCompDiffopsWithPowersMult})). It follows that
 this choice of $s$ makes all the terms in eqn
 (\ref{EqCompFTimesKillerFunction}) vanish which shows the
 kernel of $\eta$ is contained in the subset of all $f$
 killed by right multiplication of some $g\in S$.
 The other inclusion is trivial since $\eta$ is an $S$-inverting  morphism of algebras, and $f\star g=0$
 for some $g\in S$
 implies $\eta(f)\star_{S_0}\eta(g)=0$ implying
 $\eta(f)= 0$ since $\eta(g)$ is invertible in $A_{S_0}[[\lambda]]$.\\
 It is obvious that the preceding constructions can be
 done for left fractions etc. by interchanging the arguments
 in the bidifferential operators.
 This proves the Proposition since $\big(A_{S_0}[[\lambda]],\star_{S_0}\big)$ is a right (and left)
 algebra of fractions of $\big(A[[\lambda]],\star,S\big)$
 in the sense of Definition \ref{defi: bigdef}.
 \end{proof}

Note that the property $\sigma CM$ is NOT satisfied for any
`interesting' multiplicative subset $S_0$ of a \emph{Noetherian domain} $A$ where we suppose that $S_0$ does not contain $0$: we assume that there is a noninvertible
element $s_0$ in $S_0$ because otherwise both localizations
are isomorphic to the original algebra $\big(A[[\lambda]],\star\big)$. Then the sequence of principal ideals 
$\left(s_0^nA\right)_{n\in\mathbb{N}}$ clearly
equals the sequence of powers $\left(I^n\right)_{n\in\mathbb{N}}$ with $I=s_0A$, and Krull's Intersection Theorem (see e.g.~\cite[p.200, Ch.III 3.2, Corollary]{BouCommAlg}) states that 
$\cap_{n\in\mathbb{N}}s_0^nA=\{0\}$ whence for the sequence
$(s_0^n)_{n\in\mathbb{N}}$ no sequence $(a_n)_{n\in\mathbb{N}}$ can be found to satisfy property
$\sigma CM$.

\section{Non Ore multiplicative subsets in deformation quantization}
\label{SecNonOreExample}

The following example provides a multiplicative subset $S$
of a deformed algebra $\big(R=A[[\lambda]],\star\big)$ 
which is of the deformation type
$S_0+\lambda R$ (where $S_0$ is a multiplicative subset of $A$) which fails to satisfy the Ore condition, but
is a subset of a large right denominator subset of $R$.
This shows that the second problem we raised in the previous
Section \ref{SecCommutativelyLocalizedStarProducts} does not seem to be immediately trivial.

Consider
$A=\mathcal{C}^\infty(\mathbb{R}^2,\mathbb{R})$ with the
standard star product $\star$ given by formula
(\ref{EqDefStandardStarProduct}). Let $R=A[[\lambda]]$,
and let $\Omega\subset \mathbb{R}^2$ be the dense open set
of all $(x,p)\in \mathbb{R}^2$ where $x\neq 0$.
Set 
\begin{equation}
S_0:=\{1,x,x^2,\ldots\}\subset A~~\mathrm{and}~~
S:=S_0+\lambda R\subset R.
\end{equation}
Recalling the multiplicative subset $S_\Omega=\{g\in A~|~\forall~x\in\Omega:~g(x)\neq 0\}$ we have the
\begin{prop}
	The subset $S\subset R$
	is a multiplicative subset of $(R,\star)$ which is
	contained in the right denominator subset 
	$S_\Omega+\lambda R\subset R$
	(see section \ref{SecNonComLocOpenSets}), but which is neither right nor left Ore.
\end{prop}
\begin{proof}
	Since $x^m\star x^n=x^{m+n}$ it is clear that
	$S$ is a multiplicative subset of $R$ which clearly
	is a subset of
	$S_\Omega$. Next pick a smooth real-valued function
	$\chi:\mathbb{R}\to\mathbb{R}$ with the following
	properties
	\[
	\forall~p\in\mathbb{R}:~0\leq \chi(p)\leq 1,~~
	\mathrm{supp}(\chi)\subset
	\left[-\frac{1}{3},\frac{1}{3}\right],~~
	\mathrm{and~~}\forall~p\in
	\left[-\frac{1}{6},\frac{1}{6}\right]:~\chi(p)=1,
	\]
	which is well-known to exist,
	and define the smooth functions $r\in A\subset R$
	and $s\in S_0\subset S$ by
	\[
	r(x,p):=\sum_{n=0}^\infty
	\chi(p-n)\frac{(p-n)^n}{n!} 
	~~\mathrm{and}~~
	s(x,p)=x
	\]
	where $r$ is well-defined as a locally finite sum whose terms have mutually disjoint supports.
	We shall only need the following property of $r$
	which is easy to see:
	\begin{equation}\label{EqCompDerivativesOfFunnyFunctionr}
	\forall~n,k\in\mathbb{N}:~~~
	\frac{\partial^k r}{\partial p^k}(0,n)
	=\left\{ \begin{array}{cl}
	0 & \mathrm{if}~0\leq k\leq n-1, \\
	1 & \mathrm{if}~ k=n.
	\end{array}\right. .
	\end{equation}
	We remark that there are also real analytic functions
	$r:\mathbb{R}^2\to \mathbb{R}$ having the preceding
	property (\ref{EqCompDerivativesOfFunnyFunctionr}): it suffices to take the real part of the holomorphic
	function constructed by Weierstrass's elementary factors,
	see e.g.~\cite[p.~303, Thm.~15.9]{Rud87}.\\
	Next note that an element $s'\in R$ is contained in
	$S$ iff there is a unique nonnegative integer $m$ and a unique smooth function
	$g\in \lambda A[[\lambda]]$ (i.e.~$g_0=0$) such that $s'(x,p)=x^m+g$.
	For any such $s'\in S$ and $r'\in R$ we set
	\[
	\mathcal{R}(r',s'):=
	\sum_{k=0}^\infty \lambda^k\mathcal{R}_k(r',s')
	:= r \star s' - x\star r'
	\]
	which is a kind of deviation from the right Ore property for general $s'\in S$ and $r'\in R$.
	It is easy to compute that
	\[
	\forall~0\leq k\leq m:~~
	\mathcal{R}_k(r',s')(x,p)=
	{m \choose k}\frac{\partial^k r}{\partial p^k}(x,p)
	x^{m-k}
	+ \sum_{l=0}^{k-1}\frac{1}{l!}
	\frac{\partial^l r}{\partial p^l}(x,p)
	\frac{\partial^l g_{k-l}}{\partial x^l}(x,p)
	-xr_k'(x,p)
	\]
	where the empty sum (occurring for $k=0$) is defined to be $0$. Using property
	(\ref{EqCompDerivativesOfFunnyFunctionr}) it is immediate
	that
	\[
	\forall~m\in\mathbb{N},~\forall~g\in\lambda R,
	~\forall~r'\in R:~~
	\mathcal{R}_m(r',s')(0,m)=1 \neq 0,
	\mathrm{~~~hence~~~}\mathcal{R}_m(r',s')\neq 0
	\]
	showing that for the given $r\in R$, $s\in S$ there are
	no $r'\in R$ and $s'\in S$ satifying the right Ore condition. An easy application of
	eqn (\ref{EqDefNeumaierModified}) using the fact that
	$S$ is obviously stable by the bijection $V$ shows that it also fails to satisfy the left Ore condition.
\end{proof}
This example shows the difference between the general
noncommutative localization according to Proposition \ref{PGeneralLocalization} and the localization with respect to multiplicative subsets satisfying the Ore conditions,
see Theorem \ref{TLocalizationForRightDenominatorSets}:
The localization of $R$ w.r.t~$S$ exists, and its elements
are multifractions, see eqn (\ref{EqCompLocAlgebraGeneralElement}): but mapping it into the localization
with respect the the bigger Ore subset $S_\Omega+\lambda R$
helps to transform all the multifractions into simple
right (or left) fractions.

	\end{document}